\documentclass[sn-mathphys, Numbered]{sn-jnl_template}

\pdfoutput=1
\usepackage{centernot}
\usepackage{mathrsfs}

\usepackage{color, soul}

\usepackage{forest}
\usepackage{graphicx}
\usepackage{verbatim}
 \usepackage{tkz-euclide}
\usepackage[utf8]{inputenc}
\usepackage[english]{babel}
\usepackage{relsize}
\usepackage{amssymb}
\usepackage{amsthm}
\usepackage{multicol}
\usepackage{amsmath}

\usepackage{tikz-cd}
\usepackage{xcolor}
\usepackage{adjustbox}

\newtheorem{Th}{Theorem}
\newtheorem{Lem}[Th]{Lemma}
\newtheorem*{Lem*}{Lemma}
\newtheorem{Cor}[Th]{Corollary}
\newtheorem{Prop}[Th]{Proposition}

 \newtheorem*{Conj*}{Lefschetz standard conjecture $B(X)$}

\usepackage{multirow}%
\usepackage{amsmath,amssymb,amsfonts}%
\usepackage{amsthm}%
\usepackage{mathrsfs}%
\usepackage[title]{appendix}%
\usepackage{xcolor}%
\usepackage{textcomp}%
\usepackage{manyfoot}%
\usepackage{booktabs}%
\usepackage{algorithm}%
\usepackage{algorithmicx}%
\usepackage{algpseudocode}%
\usepackage{listings}

\usepackage{epsfig}
\usepackage{stmaryrd}
\newtheorem*{theorem*}{Theorem}

 \theoremstyle{definition}
 \newtheorem{Ex}{Example}
 \newtheorem*{Prob*}{Problem}
 \newtheorem*{Ex*}{Example}

 \newtheorem*{Def*}{Definition}
 \newtheorem{Rem}[Th]{Remark}

\newtheorem{Def}[Th]{Definition}

\begin{document}
\setlength{\marginparwidth}{2cm}

\title[The Lefschetz standard conjectures for generalized Kummers]{\textbf{The Lefschetz standard conjectures for IHSMs of generalized Kummer deformation type in certain degrees}}

\author*[1]{\fnm{Josiah} \sur{Foster}}\email{jrfoster@math.umass.edu}

\affil*[1]{\orgdiv{Department of Mathematics}, \orgname{University of Massachusetts Amherst}, \orgaddress{\street{710 North Pleasant Street}, \city{Amherst},  \state{MA}, \postcode{01003}, \country{United States of America}}}

\abstract{For a projective $2n$-dimensional irreducible holomorphic symplectic manifold $Y$ of generalized Kummer deformation type and $j$ the smallest prime number dividing $n+1$, we prove the Lefschetz standard conjectures in degrees $<2(n+1)(j-1)/j$. We show that the restriction homomorphism from the cohomology of a projective deformation of a moduli space of Gieseker-stable sheaves on an Abelian surface to the cohomology of $Y$ is surjective in these degrees. An immediate corollary is that the Lefschetz standard conjectures hold for $Y$ when $n+1$ is prime. The proofs rely on Markman's description of the monodromy of generalized Kummer varieties and construction of a universal family of moduli spaces of sheaves, Verbitsky's theory of hyperholomorphic sheaves, and the cohomological decomposition theorem.}

\maketitle

\section{Introduction}

\subsection{Overview}
  Let $X$ be a smooth projective variety of dimension $n$ over $\mathbb{C}$, and let the cohomological operator $L$ be given by taking the cup product with $c_1(\mathcal{O}_X(1))$. By the Hard Lefschetz theorem, there is an isomorphism
$$
L^{n-k}: H^k(X, \mathbb{Q})\xrightarrow{\sim} H^{2n-k}(X, \mathbb{Q})
$$
for all $k\leq n$. The Lefschetz standard conjecture $B(X)$ (LSC) is the following:

\begin{Conj*}(\cite[p. 196]{Gr})\label{BX}
For each $k\leq n$, there exists an algebraic self-correspondence $[\mathcal{Z}]\in H^{2k}(X\times X, \mathbb{Q})$, arising from a codimension-$k$ cycle $\mathcal{Z}\in \mathrm{CH}^k(X\times X)$, such that the isomorphism 
$$
[\mathcal{Z}]^*: H^{2n-k}(X, \mathbb{Q})\xrightarrow{\sim} H^{k}(X, \mathbb{Q}),
$$
is the inverse of the isomorphism $L^{n-k}$.
\end{Conj*}

If the correspondence $[\mathcal{Z}]$ exists for a particular $k\leq n$, the Lefschetz standard conjecture $B(X)$ is said to hold for $X$ in degree $k$, and the Lefschetz standard conjecture $B(X)$ is said to hold for $X$ if it holds in all degrees. Over $\mathbb{C}$, the standard conjecture $B(X)$ implies all the standard conjectures (see \cite{K1} and \cite{K2} for a detailed background discussion). We note also that the statement of the LSC is independent of the choice of polarization on $X$ (\cite[Proposition 4.3(1)]{K1}).

The main objective of this article is to prove the following:
\begin{Th}(Theorem \ref{proof of main theorem})\label{maintheorem}
    Let $Y$ be a projective irreducible holomorphic symplectic manifold (IHSM) of generalized Kummer deformation type of dimension $2n$. Let $j$ be the smallest prime number dividing $n+1$. Then the Lefschetz standard conjectures hold for $Y$ in degrees $<2(n+1)(j-1)/j$.
\end{Th}

When $n+1$ is prime, $j = n+1$ and we have the following immediate corollary:

\begin{Cor}\label{maincorollary}
For $Y$ a projective IHSM of generalized Kummer deformation type of dimension $2n$ for which $n+1$ is prime, the Lefschetz standard conjectures hold for $Y$.
\end{Cor}

On the other extreme, when $n+1$ is even, $j=2$ and we have the following:

\begin{Cor}\label{othermaincorollary}
For $Y$ a projective IHSM of generalized Kummer deformation type of dimension $2n$ for which $n+1\geq 4$ is even, then the Lefschetz standard conjectures hold for $Y$ in degrees $< n+1$.
\end{Cor}

If the LSC holds for a complex variety $X$, then the K\"{u}nneth components of the diagonal of $X\times X$ are algebraic, and homological and numerical equivalence coincide for algebraic cycles on $X$. The LSC is also essential for the theory of motives (see \cite[Chapter 5]{Andre}) and implies the variational Hodge conjecture (\cite[Theorem 7]{Voisin1}).
In \cite{Voisin2}, Voisin proves a series of results related to the LSC in degree $2$ for a projective irreducible holomorphic symplectic manifold that admits a Lagrangian fibration. Additionally, recent work of 
de Jong and Perry shows that the LSC in degree $2$ implies a weakened version of the period-index conjecture (\cite{dJP} Theorem 1.6). Consequently, Theorem \ref{maintheorem} adds projective IHSMs of generalized Kummer deformation type to the list of examples in \cite[Corollary 1.8]{dJP}.

In general, there are few examples for which the LSC is known to hold. It holds for varieties $X$ for which the cycle class map $Ch(X)\otimes \mathbb{Q}\rightarrow H^*(X, \mathbb{Q})$ is an isomorphism (e.g. Grassmanians, flag varieties, etc.), and it holds for Abelian varieties. The recent results of \cite{ACLS} demonstrate the LSC for IHSMs admitting a Lagrangian fibration and for particular examples of IHSMs of OG10-type. In \cite{Ar}, Arapura has proven the LSC for moduli spaces of stable vector bundles over a smooth projective curve and for the Hilbert scheme of points on any smooth projective surface, as well as for uniruled threefolds and unirational fourfolds. Furthermore, the result \cite[Corollary 7.9]{Ar} proves the LSC for the moduli space of Gieseker-stable sheaves $\mathcal{M}_H(v)$ on an Abelian surface $A$, where $v$ is a primitive, positive, and effective Mukai vector and $H$ is a $v$-generic polarization (see subsection \ref{KTheory} for definitions). A generalization of this last result (Theorem \ref{LSCforM}) is essential for the proof of Theorem \ref{maintheorem}. 

Additionally in \cite{CM}, Charles and Markman have proven the LSC for IHSMs of $K3^{[n]}$-deformation type. The following recent observation of Markman simplifies the proof of their main result:

\begin{Th}(cf. \cite[Theorem 1.5.]{Markman4}\label{Ma4, Theorem 1.5} and \cite[Section 1.2.]{Markman4}) 
Let $X$ and $Y$ be deformation equivalent irreducible holomorphic symplectic manifolds. If there exists a (twisted) derived equivalence $D^b(X)\rightarrow D^b(Y)$ with Fourier-Mukai kernel of nonzero rank, then the LSC holds for both $X$ and $Y$.
\end{Th}

It is not known in general whether for every IHSM $X$ of generalized Kummer deformation type there exists another IHSM $Y$ and a Fourier-Mukai kernel as in Theorem \ref{Ma4, Theorem 1.5}. Therefore, the results and methods found in \cite{CM} remain relevant for our present situation. However, we are optimistic that an expansion of the results in this article is possible via Theorem \ref{Ma4, Theorem 1.5}.

Let $Kum_n(A)$ denote the generalized Kummer variety associated to an Abelian surface $A$ (Definition \ref{kummer def}). Fix a lattice $\Lambda$ isometric to $H^2(Kum_n(A))$ with respect to the Beauville-Bogomolov-Fujiki and Mukai pairings (Remark \ref{BBF remark} and (\ref{cohomological Mukai pairing})) and fix a $\Lambda$-marking $$\eta: H^2(Kum_n(A)) \rightarrow \Lambda $$ (Definition \ref{marking}). The strategy of proof for Theorem \ref{maintheorem} utilizes Markman's construction in \cite{Markman1} and \cite{Markman6} of a universal family
$$
\mathcal{Y}\longrightarrow \mathfrak{M}_{\Lambda}^0
$$
 over a connected component of the moduli space of $\Lambda$-marked IHSMs (see Remark \ref{remark on moduli space}). One considers the five dimensional period domain $\Omega_{\Lambda}$ of generalized Kummer varieties, and the surjective and generically injective period map 
$$
Per: \mathfrak{M}^{0}_{\Lambda}\longrightarrow \Omega_{\Lambda}
$$
defined by sending a marked IHSM $(Y, \eta)$ to $\eta(H^{2,0}(Y))$. A key result in \cite{Markman1} yields that there exists a universal family $\mathcal{T}$ of complex tori  over the period domain $\Omega_{\Lambda}$. Pulling back along the period map $Per$ and taking the universal fiber product by the antidiagonal action (\ref{antidiagonal def}) of $\Gamma\cong A[n+1]$ provides a universal family of moduli spaces 
$$
\Tilde{\mathcal{M}}: = Per^*\mathcal{T}\times_{\underline{\Gamma}} \mathcal{Y}, 
$$
one fiber of which is the moduli space $\mathcal{M}_H(w)$ of Gieseker-stable sheaves on $A$ (see subsection \ref{Moduli Space}). The construction of the universal torus $\mathcal{T}$ is natural; there is a global isogeny between $Per^*\mathcal{T}$ and $Jac^3(\mathcal{Y})$, where the latter is the family of third intermediate Jacobians of generalized Kummer varieties \cite[Lemma 12.15]{Markman1}. We introduce these results in subsection \ref{universal deformation}. 

The main ingredients involved in proving Theorem \ref{maintheorem} are as follows: In subsection \ref{LSCM}, we prove the LSC for every projective fiber of the universal family
$$
\Pi: \Tilde{\mathcal{M}}\longrightarrow \mathfrak{M}_{\Lambda}^0 
$$
(\ref{universal family}), generalizing Arapura's result in \cite{Ar}. Furthermore, Yoshioka in \cite{Yoshioka} demonstrates that the Albanese variety of $\mathcal{M}_H(w)$ is isomorphic to $A\times \mathrm{Pic}^0(A)$ and that each fiber of the Albanese morphism
\begin{equation}\label{yoshioka albanese}
alb: \mathcal{M}_H(w)\longrightarrow A\times \mathrm{Pic}^0(A)
\end{equation}
is deformation equivalent to the generalized Kummer variety $Kum_n(A)$. We denote by $Kum_a(w)$ a fiber of the Albanese morphism (\ref{yoshioka albanese}) over $a\in A\times \mathrm{Pic}^0(A)$ and we let $$i:Kum_a(w)\hookrightarrow \mathcal{M}_H(w)$$ denote the natural inclusion morphism.
In subsection \ref{The decomposition theorem and proofs of the main results} the cohomological decomposition theorem and intersection cohomology are used to compute the degrees $i$ for which the composition of the adjoint Lefschetz correspondence with the pullback of the inclusion
$$
f: H^{2n+2-i}(\mathcal{M}_H(w), \mathbb{Q})\rightarrow H^{i}(\mathcal{M}_H(w), \mathbb{Q})\xrightarrow{i^*} H^{i}(Kum_a(w), \mathbb{Q})
$$ 
is surjective. The nondegeneracy of the Mukai pairing on the image of $f$, demonstrated in subsection \ref{restriction of the correspondence}, yields the Lefschetz standard
conjecture in these degrees $i$, hence this computation produces the degree condition in the statement of Theorem \ref{maintheorem}. 

The algebraic cycles satisfying the assumptions of the Lefschetz standard conjecture arise by deforming a particular coherent sheaf over generic twistor lines in $\mathfrak{M}_{\Lambda}^0$ (the notion of twistor lines in the moduli space of marked IHSMs is developed in \cite{Verbitsky2}). In subsection \ref{Moduli Space} we define special characteristic classes associated to these coherent sheaves. The coherent sheaves in question deform to twisted sheaves, as introduced by C\u{a}ld\u{a}raru in \cite{Caldararu}. However, intrinsic in the definition of the associated characteristic classes is an ``untwisting" designed to compensate and to ensure that the characteristic classes remain of Hodge-type along K\"{a}hler deformations in moduli. Verbitsky's theory of hyperholomorphic sheaves (\cite{Verbitsky1}, \cite{Verbitsky2}) plays a fundamental role in this strategy.

Let us give a brief outline of the organization of this article: In section \ref{Background} we collect some important background results. In particular, we relay Markman's description of the monodromy of generalized Kummer varieties and construction of a universal family of moduli spaces. In section \ref{correspondence} we construct an algebraic correspondence on every projective deformation of a moduli space of stable sheaves on an Abelian surface parametrized by the moduli space of marked IHSMs of generalized Kummer deformation type. In section \ref{LCS} we describe a decomposition of the cohomology of generalized Kummers with respect to a relative stratification and prove Theorem \ref{maintheorem}.

\section{Conventions}
\begin{enumerate}

    \item \textbf{Varieties.} Throughout this article, the term ``variety" shall refer to a separated integral scheme of finite type over the field of complex numbers $\mathbb{C}$.

    \item \textbf{Irreducible holomorphic symplectic manifolds.} The objects of study in this article are irreducible holomorphic symplectic manifolds, otherwise known as hyperk\"{a}hler manifolds. We abbreviate irreducible holomorphic symplectic manifold as IHSM.

    \item \textbf{The Lefschetz standard conjecture.} Since we always work over $\mathbb{C}$, we will throughout this article refer to the Lefschetz standard conjecture $B(X)$ (\cite[p. 196]{Gr}) as \textit{the} Lefschetz standard conjecture (LSC).
    
    \item \textbf{Rational coefficients.} Unless otherwise specified, invariants such as cohomology and the Grothendieck group will be assumed to have rational coefficients. For instance, we  write $K^0_{top}(X)$ to denote $$K^0_{top}(X,\mathbb{Z})\otimes_{\mathbb{Z}} \mathbb{Q},$$ and $H^i(X)$ to denote the rational cohomology $H^i(X, \mathbb{Q})$. 

    \item \textbf{Even and odd cohomology}
    For a variety $X$, we denote by $H^{ev}(X, \mathbb{Z})$ the ring of even cohomology of $X$. We denote by $H^{odd}(X, \mathbb{Z})$ the odd cohomology.

    \item \textbf{Derived Categories.} For a smooth projective variety $X$, we denote by $D^b(X): = D^b(Coh(X))$ the bounded derived category of coherent sheaves on $X$.

    \item \textbf{Mukai vectors.}  We will assume that the Mukai vector $$v([F]): = ch([F])\sqrt{td_X},$$ for a class $[F]\in K_{top}^0(X, \mathbb{Z})$, is primitive, positive, and of Hodge-type $(1, 1)$ (see subsection \ref{Moduli Space}). Since the Todd character of an Abelian variety $A$ is trivial, we will refer to the Mukai vector of a class $[F]\in K_{top}^0(A, \mathbb{Z})$ simply as the Chern character $ch([F])$. 

    \item \textbf{Moduli spaces of sheaves on Abelian surfaces.} We write $\mathcal{M}(w): = \mathcal{M}_H(w_{n+1})$ to denote the moduli space of $H$-generic Gieseker-stable sheaves on an Abelian surface, where $w: = w_{n+1}$ denotes the Chern character of length-$(n+1)$ zero dimensional subschemes of $A$. (The definitions are reviewed in subsection \ref{Moduli Space}.) For $v$ a Mukai vector as above, we will assume that $\mathrm{dim}(\mathcal{M}(v))\geq 8$.

    \item \textbf{Twistor lines.} We will frequently (for example in Proposition \ref{twistor family}) refer to a reduced connected projective curve $C$ within a connected component of the moduli spaces of marked generalized Kummer varieties. Such a curve is chosen to be a generic twistor line (see for example \cite[Definition 5.13]{Markman4}), but we will not utilize this fact per se.

    \item \textbf{Quotient cohomology groups.} For a smooth projective variety $X$, we define a quotient space $\overline{H}^i(X)$ of $ H^i(X)$, where $$\overline{H}^i(X): = H^{i}(X)/[R^i(X)+Alg^i(X)],$$ for which $R^i(X)$ consists of degree $i$ classes of $X$ generated by classes of degree $<i$ and $Alg^i(X)$ consists of algebraic classes of degree $i$.
\end{enumerate}

\textbf{\section{Background and preliminary results}\label{Background}}

\subsection{IHSMs of generalized Kummer type}\label{IHSMs}

  Let us define the objects of study and fix some notation.

\begin{Def}
An \textit{irreducible holomorphic symplectic manifold (IHSM)} is a simply connected compact K\"{a}hler manifold $X$, such that $H^{2,0}(X)$ is spanned by a unique (up to scale) nowhere vanishing holomorphic 2-form. 
\end{Def}
IHSMs admit hyperk\"{a}hler metrics; that is, for an IHSM of real dimension $4n$, we have a Riemannian metric with holonomy $Sp(n)$. They are therefore examples of compact hyperk\"{a}hler manifolds. Though IHSMs in general do not admit a global Torelli theorem in analogy to $K3$ surfaces, they do form a distinguished class of Calabi-Yau manifolds whose geometric properties are tightly controlled by their second cohomology (see for example \cite{Verbitsky1} and \cite{Hu3}; for a survey that describes these results in terms of Hodge-theoretic data, see \cite{Markman5}). 

\begin{Rem}\label{BBF remark}
The cohomology group $H^2(X, \mathbb{Z})$ of an IHSM $X$ of complex dimension $2n$ admits a canonical, symmetric, non-degenerate, integral, primitive bilinear pairing called the \textit{Beauville-Bogomolov-Fujiki (BBF)} pairing $q_X$ (\cite[Theorem 5]{Beauville}). For all $\alpha\in H^2(X, \mathbb{C})$, we have the equation
\begin{equation}
 \int_X \alpha^{2n} = c_Xq_X(\alpha)^n
\end{equation}
where $c_X\in \mathbb{Q}_{\geq 0}$ is the Fujiki constant. The pairing $q_X$ has signature $(3, b_2-3)$ (\cite[Theorem 5(a)]{Beauville}). This generalizes the intersection pairing on the second cohomology of $K3$ surfaces. 

\end{Rem}
Now we describe the construction of IHSMs of generalized Kummer type. Let $A$ be a complex Abelian surface and $A^{[n+1]}$ the Hilbert scheme of zero dimensional subschemes of length-$(n+1)$ on $A$. Consider
\begin{equation}\label{albanese}
A^{[n+1]}\xrightarrow{s_{n+1}} A
\end{equation}
where $s_{n+1}$ is the composition of the Hilbert-Chow morphism $h: A^{[n+1]}\rightarrow A^{(n+1)}$ sending a length-$(n+1)$ subscheme to its weighted support in $A$ and summation $\sigma: A^{(n+1)}\rightarrow A$ in the group $A$. 
\begin{Def}\label{kummer def}
The variety $Kum_n(A): = s^{-1}_{n+1}(0)$ is an IHSM, called a \textit{generalized Kummer variety} (\cite{Beauville}).
\end{Def}

The morphism $s_{n+1}$ is an isotrivial fibration, i.e. there exists a commutative diagram
\begin{center}
\begin{tikzcd}
{A\times_A A^{[n+1]}} \arrow[rrd, "\pi_2", bend left, shift left] \arrow[rdd, "\pi_1"', bend right] \arrow[rd, "\cong"] &                                                   &                                  \\
                                                                                                                        & A\times Kum_n(A) \arrow[r, "t_K"] \arrow[d, "p_1"'] & {A^{[n+1]}} \arrow[d, "s_{n+1}"] \\
                                                                                                                        & A \arrow[r, "n+1"']                                 & A                               
\end{tikzcd}                                                     
\end{center}
where $p_1$ is the projection onto the first factor and $n+1$ denotes multiplication by $n+1$. Let $t_a: A\rightarrow A$ be the translation map, defined by $t_a(x) = x+a$ for $a\in A$. The map $s_{n+1}$ is equivariant with respect to this translation action, since we let $a$ act on $A^{[n+1]}$ by $t_a$ and $a$ act on $A$ by $t_{(n+1)a}$. The isomorphism in the diagram above is given by $(a, \xi)\mapsto (a, t_{-a}(\xi))$. Therefore, the map $t_K$ is the restriction of the translation map on $A^{[n+1]}$ to $Kum_n(A)$. Since $t_a$ acts transitively on $A$, all fibers of $s_{n+1}$ are isomorphic.
\begin{Ex}
If $n = 1$, then the fiber
$s^{-1}_2(0) = Kum_1(A)$ is the Kummer surface associated to $A$, i.e. a $K3$ surface obtained from the resolution of the quotient of $A$ by the involution $\iota(x) = -x$ by blowing up at the $16$ ordinary double points. Such a $K3$ surface contains 16 $(-2)$-curves. This explains the naming convention.
\end{Ex}

\begin{Def}
If $X$ is deformation equivalent to $Kum_n(A)$, we say $X$ is of \textit{generalized Kummer type}. 
\end{Def}

\vspace{5mm}
\subsection{The Mukai lattice and $K$-theory}\label{KTheory}
For $X$ a smooth projective surface, recall we denote by $H^{ev}(X, \mathbb{Z})$ the graded ring of even cohomology of $X$.
\begin{Def}
Let $A$ be an Abelian surface. Let $K^0(A, \mathbb{Z})$ be the Grothendieck group of $A$ generated by holomorphic vector bundles on $A$, and let $K^0_{top}(A, \mathbb{Z})$ be its topological analogue. The \textit{Mukai lattice} $\Lambda(A)$ consists of the group $K^0_{top}(A, \mathbb{Z})$ along with the \textit{Mukai pairing}
\begin{equation}\label{Mukai pairing}
(x, y): = -\chi(x^{\vee}\otimes y), 
\end{equation}
for $x, y\in K^0_{top}(A, \mathbb{Z})$ and where $x^{\vee} = (x_0, -x_1, x_2)$ is the class dual to $x = (x_0, x_1, x_2)$. The Mukai pairing is given by the pushforward of $x^{\vee}\otimes y$ along the constant map $A\rightarrow \{pt\}$.
\end{Def}

\begin{Rem}
There is an isomorphism \begin{equation}\label{mukai vector isomorphism}K^0_{top}(A, \mathbb{Z})\xrightarrow[]{\sim} H^{ev}(A, \mathbb{Z})\end{equation} given by sending a class $[F]$ to its \textit{Mukai vector} \begin{equation}v([F]): = ch([F])\sqrt{td_{A}}.\end{equation} Since the Todd character of an Abelian surface is trivial, we may identify the Mukai vector of a class $[F]\in K^0_{top}(A, \mathbb{Z})$ with the Chern character of $[F]$, as per the conventions. The Chern character is given, with respect to the grading on $H^{ev}(A, \mathbb{Z})$, by $ch([F]) = (rk([F]), c_1([F]), \chi([F]))$. Moreover, by the Hirzebruch-Riemann-Roch theorem, the isomorphism (\ref{mukai vector isomorphism}) is an isometry with respect to the Mukai pairing on $K^0_{top}(A, \mathbb{Z})$ and the pairing
\begin{equation}\label{cohomological Mukai pairing}
((x_0, x_1, x_2), (y_0, y_1, y_2))_{H^{ev}} := \int_Ax_1y_1-x_0y_2-x_2y_0
\end{equation}
on $H^{ev}(A, \mathbb{Z})$.
\end{Rem}

\begin{Rem}\label{Mukai Hodge structure}
Following Mukai in \cite{Mukai2}, there is a natural weight-two Hodge structure on $H^{ev}_{\mathbb{C}}: = H^{ev}(A, \mathbb{Z})\otimes_{\mathbb{Z}}\mathbb{C}$, for which

\begin{equation}
    \begin{split}
    & H^{2,0}(H^{ev}_{\mathbb{C}}): = H^{2,0}(A)\\ & H^{0,2}(H^{ev}_{\mathbb{C}}): = H^{0,2}(A)\\
    & H^{1, 1}(H^{ev}_{\mathbb{C}}): = H^0(A)\oplus H^{1,1}(A)\oplus H^4(A).\\
    \end{split}
\end{equation}
\end{Rem}

Since $H^1(A, \mathbb{Z})\neq 0$, we need also to consider the odd $K$-groups of both $A$ and, eventually, moduli spaces of sheaves on $A$ (see subsection \ref{Moduli Space}). We therefore recall the necessary construction in $K$-theory following \cite{At} and \cite{AH}.

\begin{Def}
Let $X$ be a complex algebraic variety. The \textit{reduced} $K$-group $\tilde{K}^0_{top}(X)$ is the kernel of the restriction homomorphism $K^0_{top}(X)\rightarrow K^0_{top}(x_0)$ for $x_0\in X$. 
\end{Def}

\begin{Def}
Denote by $SX$ the topological suspension of $X$ and define the odd $K$-group as $$K^1_{top}(X, \mathbb{Z}):= \tilde{K}^0(SX, \mathbb{Z}),$$ and $K^1_{top}(X): = K^1_{top}(X,\mathbb{Z})\otimes \mathbb{Q}$ (cf. \cite[pg. 199]{AH}). We define the \textit{$K$-ring} as

$$K^{\bullet}_{top}(X): = K^0_{top}(X)\oplus K^1_{top}(X).$$
Define the reduced cohomology group $\tilde{H}^*(X, \mathbb{Q})$ analogously to the reduced $K$-group.
\end{Def}

We have, in addition, the suspension isomorphism $$\sigma^n: \Tilde{H}^*(X, \mathbb{Q})\xrightarrow[]{\sim} \Tilde{H}^*(S^n(X), \mathbb{Q})$$ given by taking the cup product of $a\in \Tilde{H}^*(X, \mathbb{Q})$ with the canonical generator of $H^n(S^n, \mathbb{Z})$.

Suppose that $X$ and $Y$ have the structure of finite $CW$-complexes and assume that $H^*(X, \mathbb{Z})$ is torsion-free. By the K\"{u}nneth theorem (see \cite[Theorem 1]{At}), there is an isomorphism

\begin{equation}\label{KKunneth}
[K^0_{top}(X)\otimes K^0_{top}(Y)]\oplus [K^1_{top}(X)\otimes K^1_{top}(Y)]\xrightarrow[]{\sim} K^0_{top}(X\times Y).
\end{equation}

We can now define the odd Chern classes in terms of the odd $K$-group: 
\begin{Def}\label{odd Chern classes}
Let $x$ be a class in $K^1_{top}(X)$, for $X$ a complex projective variety. We denote by $\Tilde{x}$ the corresponding class in $\Tilde{K}^0(SX^+)$, where for $x_0\in X$, $SX^+$ denotes the suspension of $X^+: = X\sqcup \{x_0\}$. Let $i$ be a half-integer, i.e. a rational number such that $2i$ is an odd integer. The Chern class $c_i(x)$ of $x$ is defined as the image in $H^{2i}(X, \mathbb{Z})$ of $c_{i+\frac{1}{2}}(\Tilde{x})$, via the isomorphism $H^{2i}(X, \mathbb{Z})\cong H^{2i+1}(SX^+, \mathbb{Z})$.
\end{Def}

\begin{Lem}(\cite[pg. 209]{AH})\label{AH isomorphism}
There is a ring isomorphism 
$$
ch: K^{\bullet}_{top}(X)\rightarrow H^*(X, \mathbb{Q})
$$
which sends $K^0_{top}(X)$ to $H^{ev}(X, \mathbb{Q})$ and $K^1_{top}(X)$ to $H^{odd}(X, \mathbb{Q})$. 
\end{Lem}

\begin{Rem}
It is necessary for us to generalize the Mukai pairing to a pairing on $K^{\bullet}(A)$. Define the operator $\tau: H^i(A, \mathbb{Z})\rightarrow H^i(A, \mathbb{Z})$ where $\tau$ acts by $(-1)^{i(i-1)/2}$. Note that $\tau$ restricts to a map on $H^2(A, \mathbb{Z})$ via $\tau(x) = -x$. The generalized Mukai pairing on $H^*(A, \mathbb{Z})$ may be expressed as 
\begin{equation}\label{Generalized Mukai pairing}
( v, w ) : = -\int_A \tau(v)\cup w.
\end{equation}
The above pairing pulls back to the pairing (\ref{Mukai pairing}) via the Chern character map.
\end{Rem}
\vspace{5mm}

\subsection{Moduli spaces of sheaves on an Abelian surface and the $\theta$-classes}\label{Moduli Space}

The purpose of this subsection is to introduce certain characteristic classes on moduli spaces of sheaves on an Abelian surface. After relating the cohomology rings of projective deformations of moduli spaces of sheaves to the cohomology rings of projective deformations of generalized Kummer varieties, such characteristic classes will be our source of algebraic cycles in the proof of the LSC.  

 \begin{Def}
 We say a Mukai vector $v = ch([F])\in H^{ev}(A, \mathbb{Z})$ is \textit{primitive} if it is not a multiple of some other class in $H^{ev}(A, \mathbb{Z})$.
 \end{Def}
 \begin{Rem}
  Let $F$ be a Gieseker-stable sheaf on $A$, for which $v = ch([F])$ is primitive and $c_1([F])$ is of Hodge-type (1,1). There is a system of hyperplanes in the ample cone $Amp(A)$ called $v$-walls, which is countable but locally finite (see \cite[section 4C]{HL}).
  \end{Rem}

 \begin{Def}
 A class $H\in Amp(A)$ is said to be a \textit{$v$-generic polarization} if it does not belong to any $v$-wall.
 \end{Def}
 
 Let $\mathcal{M}(v): = \mathcal{M}_H(v)$ denote the moduli space of Gieseker-stable sheaves with Chern character $v$ and $H$ a $v$-generic polarization. It has been shown in \cite{Mukai1} that $\mathcal{M}_H(v)$ is a smooth and symplectic quasi-projective variety and that dim$(\mathcal{M}(v)) = 2+(v^{\vee}, v)$, where $v^{\vee}$ is the class dual to $v$.
 
Identifying $H^4(A, \mathbb{Z})$ with $\mathbb{Z}$, we may write $\chi(F) = a\omega$, where $\omega$ is the fundamental class of $A$. We make the following definition following \cite{Yoshioka}:
 
 \begin{Def}
 A Mukai vector $v := ch(F) = (rk(F),\, c_1(F),\, a\omega) \in H^{ev}(A, \mathbb{Z})$ is \textit{positive} if one of the following cases holds: 
 \begin{enumerate}
     \item $rk(F)>0$.
     \item $rk(F) = 0$, $c_1(F)$ is effective, and $a\neq 0$.
     \item $rk(F) = c_1(F) = 0$ and $a<0$.
 \end{enumerate}
 \end{Def}
 
  Let $v$ be primitive, positive, and of Hodge-type $(1, 1)$ with respect to Mukai's Hodge structure on $H^{ev}(A, \mathbb{C}$), and assume that dim$(\mathcal{M}_H(v))\geq 4$. Then the moduli space $\mathcal{M}_H(v)$ is deformation equivalent to $A^{[n+1]}\times \mathrm{Pic}^0(A)$ (\cite[Theorem 0.1]{Yoshioka}).
 
Let $w_{n+1}$ denote the Chern character of the ideal sheaf of a length $(n+1)$-subscheme of $A$ and let $\mathcal{M}_H(w_{n+1})$ be the moduli space of rank one torsion-free sheaves on $A$ with Chern character $w_{n+1}$. As in the conventions, we will identify $w: = w_{n+1}$ and write $\mathcal{M}(w)$ for the moduli space $\mathcal{M}_H(w_{n+1})$.

\begin{Def}
A \textit{universal sheaf} is a coherent sheaf $\mathcal{E}$ over $A\times \mathcal{M}(v)$, flat over $\mathcal{M}(v)$, whose restriction to $A\times \{m\}$, $m\in \mathcal{M}(v)$, is isomorphic to $F_m$, for $F_m$ a sheaf representing the isomorphism class $m$. 
\end{Def}

\begin{Rem}
If there exists a class $x\in K^0(A)$ such that $\chi(x\otimes v) = 1$, then a universal sheaf $\mathcal{E}$ exists. The moduli space $\mathcal{M}(v)$ always admits a (twisted) universal sheaf (see the appendix of \cite{Mukai2}).
\end{Rem}

  Let $\mathcal{E}$ be a universal sheaf over $A\times \mathcal{M}(v)$. Let $\pi_{ij}$ be the projection of $\mathcal{M}(v)\times A\times \mathcal{M}(v)$ to the $i$th and $j$th factors. Consider the following element in the bounded derived category $D^b(\mathcal{M}(v)\times \mathcal{M}(v)$):
  \begin{equation}\label{E bullet}
  E^{\bullet}: = R\pi_{13, *}R\mathcal{H}om(\pi^*_{12}\mathcal{E}, \pi^*_{23}\mathcal{E}).
  \end{equation}
Define
\begin{equation}\label{RelativeExtension}
E^i: = \mathcal{E}xt_{\pi_{13}}^i(\pi_{12}^*\mathcal{E}, \pi_{23}^*\mathcal{E})
\end{equation}
to be the $i$th relative extension sheaf and let 
\begin{equation}\label{KclassRelativeExtension}
    [E^{\bullet}]: = \sum_{i}(-1)^i[\mathcal{E}xt^i_{\pi_{13}}(\pi_{12}^*\mathcal{E}, \pi_{23}^*\mathcal{E})] = [E^2]-[E^1]
\end{equation}    
denote the class of $E^{\bullet}$ in the Grothendieck group of $\mathcal{M}(v)\times \mathcal{M}(v)$. We have that $E^0$ vanishes, while $E^2$ is isomorphic to the structure sheaf of the diagonal of $\mathcal{M}(v)\times \mathcal{M}(v)$. Furthermore, the sheaf $E^1$ is a reflexive torsion-free sheaf of rank $(v, v) = 2n+2$ and is locally free away from the diagonal, by \cite[Remark 4.6]{Markman2}.

\begin{Rem}\label{decomposition of first chern class}
Let $p_i$ denote the projection of $\mathcal{M}(v)\times \mathcal{M}(v)$ onto the $i$th factor. There is a canonical decomposition 
$$
c_1([E^{\bullet}]) = -c_1([E^1]) = p_1^*\alpha + p_2^*\beta + \delta
$$
for which $p_1^*\alpha \in H^2(\mathcal{M}(v))\otimes H^0(\mathcal{M}(v))$, $p_2^*\beta \in H^0(\mathcal{M}(v))\otimes H^2(\mathcal{M}(v))$ and \begin{equation}\delta  = \sum_ip_1^*\delta^{\prime}_i\wedge p_2^*\delta^{\prime\prime}_i \in H^1(\mathcal{M}(v))\otimes H^1(\mathcal{M}(v)).\end{equation}
\end{Rem}
\noindent Define characteristic classes $\theta([E^{\bullet}])$ and $\theta([E^1])$ by 
\begin{equation}\label{theta class 1}
    \theta([E^{\bullet}]): = ch([E^{\bullet}])\mathrm{exp}\left(\frac{-p^*_1\alpha - p_2^*\beta}{2n+2}\right)
\end{equation} and
\begin{equation}\label{theta class 2}
    \theta([E^1]): = ch([E^1])\mathrm{exp}\left(\frac{-p^*_1\alpha - p_2^*\beta}{2n+2}\right).
    \end{equation}
We denote by $\theta_i([E^{\bullet}])$ and $\theta_i([E^1])$ the $i$th summands of (\ref{theta class 1}) and (\ref{theta class 2}) respectively in $H^{2i}(\mathcal{M}(v)\times \mathcal{M}(v), \mathbb{Q})$.

\begin{Rem}\label{description of the kappa class}
Note the close relationship between the characteristic class $\theta([E^{1}])$ and the characteristic class 
$$
\kappa([E^{1}]): = ch([E^{1}])\mathrm{exp}\left(\frac{-c_1([E^{1}])}{2n+2}\right)
$$
defined originally for IHSMs of $K3^{[n]}$-deformation type in \cite{Markman2}. For $X$ a topological space and $[F]$ a class in $K^{\bullet}_{top}(X)$ of rank $r>0$, let $x_j$ denote the Chern roots of $[F]$. The summands $\kappa_i([F])$ are defined as 
$$
\kappa_i([F]): = \left(\sum_{j = 1}^r\frac{x_j-\left(\frac{\sum_{k = 1}^r x_k}{r}\right)}{i!}\right)^i.
$$
\end{Rem}
By construction, $\kappa_i([E^{1}])$ is algebraic for each $i$ as a class in $H^{2i}(\mathcal{M}(v)\times \mathcal{M}(v), \mathbb{Q})$. 
Note that there is the following relationship between the classes $\kappa_i([E^{1}])$ and $\theta_i([E^{1}])$: 
$$
\kappa_i([E^{1}]) = \theta_i([E^{1}])\left(\mathrm{exp}\left(\frac{-\delta}{2n+2}\right)\right),
$$
for $\delta$ in Remark \ref{decomposition of first chern class}. The algebraicity of $\theta_i([E^1])$ is deduced from the algebraicity of $\kappa_i([E^1])$, as the classes differ by taking the cup product with an algebraic class.

\begin{Rem}\label{kappa class is analytic}
For $[\tilde{F}]$ the class of a twisted coherent sheaf (see \cite{Caldararu} in addition to \cite[Section 2.2]{Markman2} for a discussion on twisted sheaves), Markman defines the class $\kappa([\tilde{F}])$ as the Taylor series of the branch of the $r$-th root function \cite[Equation 2.6]{Markman2}:
$$
\kappa([\tilde{F}]): = Sqrt_r(ch([\tilde{F}^{\otimes r}\otimes \mathrm{det}([\tilde{F}^{-1}]])).
$$
We will later utilize this definition to conclude that the classes $\kappa_i([E^{1}_X])$  and $\theta_i([E^1_X])$ are algebraic, where $[E^{1}_X]$ is the twisted coherent sheaf obtained from $[E^{1}]$ via a projective deformation of a product of moduli spaces of stable sheaves on an Abelian surface (we provide details in subsection \ref{universal deformation}). Here we rely fundamentally on Verbitsky's theory of hyperholomorphic sheaves, \cite{Verbitsky1}. 
\end{Rem}
\vspace{5mm}
\subsection{Monodromy of generalized Kummer varieties}\label{monodromy}

The main result we prove in this subsection states that the previously described $\theta$-classes (\ref{theta class 2}) are invariant under the diagonal monodromy action. To put this in the proper context, we need to introduce results from \cite{Markman1}. See in particular \cite[Section 4]{Markman1}. For background results on Clifford algebras, we recommend \cite{Chevalley}.
\begin{Def}
Let $X$ be a smooth projective variety and let $g$ be an automorphism of $H^*(X, \mathbb{Z})$. If there exists a family $\mathcal{X}\rightarrow B$ of compact K\"{a}hler manifolds such that $X$ is the fiber over $b_0$ and if $g$ is in the image of $\pi_1(B, b_0)$ under the monodromy representation, then we say $g$ is a \textit{monodromy operator}. The subgroup $\mathrm{Mon}(X)\subset \mathrm{GL}(H^*(X, \mathbb{Z}))$ generated by all monodromy operators is called the \textit{monodromy group} of $X$.

\end{Def}

For $A$ an Abelian surface, let
$$
V: = H^1(A, \mathbb{Z})\oplus H^1(A,\mathbb{Z})^*,
$$
where we naturally identify $H^1(A, \mathbb{Z})^*$ with $H^1(\mathrm{Pic}^0(A), \mathbb{Z})$. We associate to $V$ a bilinear form $(\cdot, \cdot)_{V}$ given by
\begin{equation}\label{bilinear}
((\alpha_1, \omega_1), (\alpha_2, \omega_2))_{V}: = \omega_2(\alpha_1) + \omega_1(\alpha_2). 
\end{equation}
Then, $V$ is a lattice isometric to the direct sum of four copies of the hyperbolic plane.

The character group of $SO(V)$ is isomorphic to $\mathbb{Z}/2\mathbb{Z}$. It is generated by the orientation (or norm) character
$$
ort: SO(V)\rightarrow \{\pm 1\}.
$$
The character $ort$ is generated by the action of isometries of the top cohomology of the positive cone in $V\otimes_{\mathbb{Z}} \mathbb{R}$ (cf. \cite[section 4]{Markman5}).
Denote by $SO_{+}(V)$ the kernel of $ort$. Up to isomorphism, $\mathrm{Spin}(V): = \mathrm{Spin}(V, (\cdot, \cdot)_{V})$ is the unique double cover of the index two subgroup $SO_+(V)$ of the the special orthogonal group $SO(V, (\cdot, \cdot)_{V})$. The cohomology ring $H^*(A, \mathbb{Z})$ is the spin representation of the arithmetic group $\mathrm{Spin}(V)$.

It is convenient for us to understand the monodromy in the language of Clifford algebras. Recall that the Clifford algebra $C(V)$ is the quotient of the tensor algebra $\oplus_{n\geq 0} V^{\otimes n}$ by the relation 
$$
x\cdot x^{\prime} + x^{\prime}\cdot x = (x, x^{\prime})_{V}.
$$

\begin{Def}
 Set $G(V): = \{x\in C(V)^{\times}: xVx^{-1}\subset V\}.$ The standard representation $V$ is defined by the homomorphism
 \begin{equation}
 \begin{split}
     &\rho: G(V)\rightarrow O(V)\\
     &\rho_x(v) = x\cdot v\cdot x^{-1}.
     \end{split}
 \end{equation}
 
\end{Def}

The group $H^*(A, \mathbb{Z})$ is isomorphic to $\wedge^{*}H^1(A, \mathbb{Z})$ and $H^*(A, \mathbb{Z})$ has a $C(V)$-module structure as follows: $V$ embeds into $C(V)$ and $V$ embeds into $\mathrm{End}[H^*(A, \mathbb{Z})]$, where for $(w, \theta)\in V$, the homomorphism 
\begin{equation}
    m: V\rightarrow \mathrm{End}[H^*(A, \mathbb{Z})]
\end{equation}
is defined by sending $w$ to $w\wedge (\bullet)$ and $\theta$ to the derivation $D_{\theta}$ sending $H^i(A, \mathbb{Z})$ to $H^{i-1}(A, \mathbb{Z})$. The homomorphism $m$ satisfies the Clifford relation, i.e. $m(v)\circ m(w) + m(w)\circ m(v) = (v, w)id_{H^*(A, \mathbb{Z})}$ (cf. \cite[Section 4.1]{Markman1}). The Clifford product defines the homomorphisms $V\otimes S^+\rightarrow S^-$, $V\otimes S^-\rightarrow S^+$. We define the homomorphism $S^{+}\otimes S^-\rightarrow V$ as the composition of the homomorphism $V^*\rightarrow V$ given by the pairing $(\cdot, \cdot)_V$ and the homomorphism $S^+\otimes S^-\rightarrow V$ sending $s\otimes t$ to $((\cdot)s, t)_{S^-}$, where $(\cdot, \cdot)_{S^-}$ is the restriction of the Mukai pairing (\ref{Generalized Mukai pairing}) to $S^-$. The homomorphism $m_{w}$ corresponds to Clifford multplication $(\bullet)\cdot w$ with the Chern character $w = (1, 0, -n-1)\in S^{+}$ of a length-(n+1) subscheme of $A$. We have for $y = (v, \theta) \in V$ 
$$
m_{w}(y) = v\wedge w+D_{\theta}(w).
$$
The cokernel of $m_{w}$ is isomorphic to the the group $\Gamma$ of points of order $n+1$ on $A$ (see the calculation in \cite[pp. 16-17]{Markman1}). Let $m_{w}^{\dagger}: S^{-}\rightarrow V$ be the adjoint of $m_{w}$. Given a primitive element $v\in S^{+}$ such that $(v, v)_{S^+} = -2n$, there exists an element $g\in \mathrm{Spin}(V)$ such that $g(w) = v$. For $x\in V$ and considering $x$ and $g$ as elements of $C(V)$, there are equalities
$$
g(m_{w}(x)) = g\cdot x\cdot w = (g\cdot x\cdot g^{-1})\cdot (g\cdot w) = m_v(\rho_g(x)).
$$
Let $m_v^{\dagger}$ be the adjoint of $m_v$. There is a commutative diagram
\begin{center}
\begin{tikzcd}
V \arrow[d, "\rho(g)"'] \arrow[r, "m_{w}"] & S^- \arrow[r, "m_{w}^{\dagger}"] \arrow[d, "g"] & V \arrow[d, "\rho(g)"'] \\
V \arrow[r, "m_v"']                              & S^- \arrow[r, "m_v^{\dagger}"']                       & V                      
\end{tikzcd}
\end{center}

The cokernels of $m_{w}$ and $m_v$ are isomorphic. Let $m^{-1}_v: S^-_{\mathbb{Q}}\rightarrow V_{\mathbb{Q}}$ be the inverse of $m_{v}: V_{\mathbb{Q}}\rightarrow S^-_{\mathbb{Q}}$. Then the quotient 
\begin{equation}\label{Gamma_v}
    m^{-1}_v(S^-)/V \cong \Gamma_v
\end{equation}
is a subgroup of the compact torus $V_{\mathbb{R}}/V$ isomorphic to $(\mathbb{Z}/(n+1)\mathbb{Z})^4$ and canonically associated to $v$.

\begin{Rem}\label{isomorphism of gamma and torsion points}
For $w$ the Chern character of the ideal sheaf of a length-$(n+1)$ subscheme of $A$, there is a commutative diagram
\begin{center}
    \begin{tikzcd}
\Gamma_w \arrow[d, "\cong"] \arrow[r, hook] & A\times \mathrm{Pic}^0(A)               \\
{A[n+1]} \arrow[r, hook]                    & A\times \{0\} \arrow[u, hook]
\end{tikzcd}
\end{center}
The group $A[n+1]$ of torsion points of $A$ of order $n+1$ acts on $A^{[n+1]}$. The Albanese map (\ref{albanese}) is invariant with respect to this action. Hence, $A[n+1]$ embeds into $\mathrm{Aut}_0(Kum_n(A))$, where $\mathrm{Aut}_0(Kum_n(A))$ is the group of automorphisms of $Kum_n(A)$ acting trivially on $H^2(Kum_n(A))$.
\end{Rem}

The following lemma states that $\mathrm{Aut}_0(Kum_n(A))$ deforms as a local system in the moduli space of marked generalized Kummers (see \cite{Hassett Tschinkel} for details):

\begin{Lem}(\cite[Theorem 2.1, and Proposition 3.1]{Hassett Tschinkel})\label{deformation invariant}
Let $n>1$. $Aut_0(Kum_n(A))$ is a deformation invariant of $Kum_n(A)$. The action of $\Gamma_w$ on $A$ extends to an action on $Kum_n(A)$ and furthermore extends to a natural action on any deformation of $Kum_n(A)$. 
\end{Lem}

Let $\mathcal{P}$ be an object in $D^b(A\times A)$ such that $\mathcal{P}$ is the Fourier-Mukai kernel of an autoequivalence $\Phi_{\mathcal{P}}: D^b(A)\rightarrow D^b(A)$. If we consider $ch(\mathcal{P})$ as a correspondence, it induces an automorphism on $H^*(A, \mathbb{Z})$, considered as a spin representation, which is the image of an element $g\in \mathrm{Spin}(V)$ in $\mathrm{GL}(H^*(A, \mathbb{Z}))$ (see for example \cite{Orlov}). 

Following \cite{Markman1}, there is a correspondence $\gamma_g$ associated to $g\in \mathrm{Spin}(V)$ as follows: Let $\mathcal{E}$ be a universal sheaf on $A\times \mathcal{M}(w)$ and let $\pi_{ij}$ denote the projection of $A\times \mathcal{M}(w)\times A\times \mathcal{M}(w)$ onto the $i$th and $j$th factors. Define

\begin{equation}\label{gamma}
\gamma_{g}: = c_{2n+4}(\pi_{24, *}[\pi_{12}^*\mathcal{E}^*\otimes \pi_{34}^*\mathcal{E}\otimes \pi_{13}^*\mathcal{P}][1])
\end{equation}

\noindent where the pushforward, pullback, and dual are taken in the derived category. By the Grothendieck-Riemann-Roch theorem, one can express $\gamma_{g}$ in terms of $ch(\mathcal{P})$. When one identifies $ch(\mathcal{P})$ with $g$ in this expression, there arises a class $$\gamma_g\in H^{4n+8}(\mathcal{M}(w)\times \mathcal{M}(w), \mathbb{Z})$$ for each $g\in \mathrm{Spin}(V)$. 
\begin{Def}
Define $\mathrm{Spin}(V)_{v}$ to be the subgroup of Spin$(V)$ that stabilizes $v$. 
\end{Def}
The following lemma yields a simplified version of several results in \cite{Markman1}. A full description of the interplay between derived autoequivalences of an Abelian surface $A$ and the monodromy of $\mathcal{M}(w)$ is present there.

\begin{Lem}
The correspondence $\gamma_g: H^*(\mathcal{M}(w), \mathbb{Z})\longrightarrow  H^*(\mathcal{M}(w), \mathbb{Z})$ given by (\ref{gamma}) is a graded ring homomorphism, and there is an induced homomorphism
\begin{equation}\label{mon g}
\mathrm{mon}_g: \mathrm{Spin}(V)_{w}\longrightarrow \mathrm{GL}(H^*(\mathcal{M}(w), \mathbb{Z}))
\end{equation}
for each $g\in \mathrm{Spin}(V)_w$.
\end{Lem}
\begin{proof}
    Indeed, the result has been shown in \cite[Corollary 9.4]{Markman1} for the group $$G^{ev}_w\subset GL(V\oplus S^+\oplus S^-)$$ where $G^{ev}$ is a Clifford group generated by $\mathrm{Spin}(V)$ and an involution. We hence observe that the image of $\mathrm{Spin}(V)_w$ under $\mathrm{mon}_g$ is contained in the monodromy group of $\mathcal{M}(w)$. 
\end{proof}

 Write $S^+: = H^{ev}(A, \mathbb{Z})$ and $S^- : = H^{odd}(A, \mathbb{Z})$. Then Spin$(V)$ acts on $V$, $S^+$, and $S^-$. 
Let
\vspace{3mm}
\begin{equation}
\begin{split}
& S^{+}_{\mathbb{Q}}: = S^+\otimes_{\mathbb{Z}} \mathbb{Q},\\
& S^{-}_{\mathbb{Q}}: = S^{-}\otimes_{\mathbb{Z}}\mathbb{Q},\\ 
& S_A: = H^*(A, \mathbb{Q}).\\
\end{split} 
\end{equation}
\vspace{3mm}
There is a homomorphism $\tilde{\xi}: S_A\rightarrow H^*(\mathcal{M}(w), \mathbb{Q})$, defined by
\begin{equation}\label{xi hom}
\tilde{\xi}(\lambda): = \pi_{\mathcal{M}, *}[\pi^*_A(\tau_A(\lambda))\cup ch(\mathcal{E})],
\end{equation}
for $\mathcal{E}$ a universal sheaf over $A\times \mathcal{M}(w)$. Define $S^1_A: = S^-$, $S^2_A: = S^+\cap w^{\perp}$, and for $j>2$, define $S^j_A: = S^+$ if $j$ is even, and $S^j_A: = S^-$ if $j$ is odd. There is a homomorphism 
$$
\tilde{\xi}_j: S^j_A\rightarrow H^*(\mathcal{M}(w), \mathbb{Q})
$$
for each $j$, which is induced by the inclusion $S^j_A\subset S_A$.

The next lemma gives a useful description of the cohomology of $\mathcal{M}(w)$:

\begin{Lem}\label{cohomology ring of M}
Let $T$ denote the Abelian fourfold $A\times \mathrm{Pic}^0(A)$. Then 

$$H^*(\mathcal{M}(w)) = H^*(T)\otimes H^*(Kum_n(A))^{\Gamma_w},$$
where $H^*(Kum_n(A))^{\Gamma_w}$ denotes the subalgebra of $H^*(Kum_n(A))$ invariant under translation by $\Gamma_w$.
\end{Lem}
\begin{proof}
The two main theorems of \cite{Yoshioka} yield that each fiber of the Albanese morphism $\mathcal{M}_H(v)\rightarrow A\times \mathrm{Pic}^0(A)$ (\ref{yoshioka albanese}) is of generalized Kummer type. By \cite[Lemma 10.1(3)]{Markman1}, for $v$ a primitive Mukai vector and $Kum_a(v)$ the fiber over $a$ of the Albanese map, $\mathcal{M}(v)$ is the quotient of $T\times Kum_a(v)$ by the anti-diagonal action $\tilde{\Gamma}_v$ of $\Gamma_{v}$ (see (\ref{Gamma_v}); this fact is referenced later as Lemma \ref{Markman 10.1(3)}). The action of $\Gamma_v$ on $H^*(T)$ is trivial, hence, $$H^*(T\times Kum_a(v)/\tilde{\Gamma}_{v}) = H^*(T)\otimes H^*(Kum_a(v))^{\Gamma_{v}}.$$
\end{proof}

\begin{Lem}\label{trivial character}
The trivial character of the representation 

$$\mathrm{mon}_g: \mathrm{Spin}(V)_w\rightarrow \mathrm{GL}(H^2(\mathcal{M}(w), \mathbb{Q}))
$$
arising from (\ref{mon g}) appears with multiplicity zero.
\end{Lem}
\begin{proof}
By the K\"{u}nneth theorem and Lemma \ref{cohomology ring of M},

$$
H^2(\mathcal{M}(w)) = \sum_{i+j = 2} H^i(T)\otimes H^j(Kum_n(A))^{\Gamma_w},
$$
where by definition, $H^1(T, \mathbb{Z}) \cong V$ as $\mathrm{Spin}(V)$-representations and so $\mathrm{Spin}(V)_w$ acts on $H^2(T, \mathbb{Z}) = \wedge^2 V$. The $8$-dimensional standard representation of $\mathrm{Spin}(7)$ is irreducible, while at the same time the $\mathrm{Spin}(7)$-invariant subgroup of $\wedge^2 V$ occurs with multiplicity zero by \cite[Sec. 2.1.]{Munoz}. Note that $H^1(Kum_n(A))^{\Gamma_w} = 0$. We proceed by proving the lemma for $H^2(Kum_n(A))^{\Gamma_w}$. We have the equality $$H^2(Kum_n(A))^{\Gamma_w} = H^2(Kum_n(A))$$ by \cite[Lemma 10.1(4)]{Markman1}. The subspace $w^{\perp}_{\mathbb{Q}}\subset H^{ev}(A, \mathbb{Q})$ is an irreducible $\mathrm{Spin}(V)_w$-representation. We establish the $\mathrm{Spin}(V)_w$-equivariance properties of Yoshioka's Hodge isometry
\begin{equation}\label{YHI}
\xi: w^{\perp}\longrightarrow H^2(Kum_{n}(A), \mathbb{Z}),
\end{equation}
\cite[Theorem 0.2]{Yoshioka}. Note that since $H^1(Kum_{n}(A), \mathbb{Z})$, vanishes, the above isomorphism factors through the homomorphism
$$
\tilde{\xi}_2: w^{\perp}\cap S^{+}_A\longrightarrow H^2(Kum_{n}(A), \mathbb{Q}),
$$
given by (\ref{xi hom}).
We check $\tilde{\xi}_2$ is $\mathrm{mon}_g$-equivariant with respect to the action of $\mathrm{Spin}(V)_w$: We first observe that the image of $\tilde{\xi}_2$ generates $H^2(Kum_a(w), \mathbb{Q})$, as the K\"{u}nneth factors of $ch(\mathcal{E})$ generate $H^*(\mathcal{M}(w), \mathbb{Q})$ by \cite[Corollary 2]{Markman3}. For all $\lambda\in S^2_A$ and all $g\in \mathrm{Spin}(V)_{w}$ there is the equality
    $$
    \mathrm{mon}_g(\tilde{\xi}_2(\lambda)) = \tau^{ort(g)}_{\mathcal{M}}[\tilde{\xi}_2(\tau^{ort(g)}g\tau^{ort(g)}(\lambda))],
    $$
    as a special case of \cite[Corollary 8.7]{Markman1}, where $\tau$ acts on the $i$th cohomology by $(-1)^{i(i-1)/2}$. In particular, $\tilde{\xi}_2$ is $\mathrm{mon}$-equivariant with respect to the $\mathrm{Spin}(V)_w$-action.

\end{proof}
The monodromy equivariance of Yoshioka's Hodge isometry (\ref{YHI}) implies there is a direct sum decomposition 
    $$
    H^*(A) \cong \mathbb{Q}w\oplus w^{\perp}_{\mathbb{Q}}\oplus H^{odd}(A)
    $$
    of irreducible $\mathrm{Spin}(V)_w$-representations. More generally, Lemma \ref{trivial character} implies that, for $v$ the Chern character of a coherent sheaf on $A$, there is an orthogonal direct sum decomposition 
$$
S^+_{\mathbb{Q}} = \mathbb{Q}v\oplus v^{\perp}_{\mathbb{Q}}
$$

\noindent is a direct sum decomposition into irreducible $\mathrm{Spin}(V)_v$-representations. Then $S_{\mathbb{Q}}^- = H^{odd}(A)$ is also an irreducible $\mathrm{Spin}(V)_v$-representation. Identifying $H^*(A)$ with $K^{\bullet}(A)$ via the isomorphism given by the Chern character (see Lemma \ref{AH isomorphism}) yields an analogous decomposition of $K^{\bullet}(A)$:

\begin{Lem}\label{Decomposition of G-representations}
Let $v\in \Lambda(A)$ be a Chern character of a locally free sheaf on $A$. We can express $K^{\bullet}(A)$ as the direct sum of irreducible $\mathrm{Spin}(V)_v$-representations $\mathbb{Q}v\oplus v^{\perp}_{\mathbb{Q}}\oplus K^1(A)$.
\end{Lem}

\begin{Prop}\label{monodromy invariance of the theta class}
Let $\theta([E^1])$ be the characteristic class defined in (\ref{theta class 2}). Then for each $i$,

$$\theta_i([E^1])\in H^{2i}(\mathcal{M}(w)\times \mathcal{M}(w), \mathbb{Q})$$ is invariant under the diagonal monodromy action.  
\end{Prop}
\begin{proof}
We utilize the relationship between the characteristic classes $\theta$ and $\kappa$. As described in Remark \ref{description of the kappa class}, the latter is defined as
$$
\kappa([E^{\bullet}]): = ch([E^{\bullet}])\mathrm{exp}\left(\frac{-c_1([E^{\bullet}])}{2n+2}\right).
$$
Recall that the element
$$
E^{\bullet}: = R\pi_{13, *}R\mathcal{H}om(\pi_{13}^*\mathcal{E}, \pi_{23}^*\mathcal{E}) 
$$
in $D^b(\mathcal{M}(w)\times \mathcal{M}(w))$ was the class defined in (\ref{E bullet}), where $\mathcal{E}$ is a universal sheaf over $A\times \mathcal{M}(w)$. The first sheaf cohomology of $E^{\bullet}$ is $E^1$ and the second is $E^2$. Its corresponding element in $K^{\bullet}(\mathcal{M}(w)\times \mathcal{M}(w))$ is $[E^{\bullet}]$. We obtain the Chern character of $E^{\bullet}$ by contracting the tensor product $ch(\mathcal{E})^{\vee}\otimes ch(\mathcal{E})$ in $H^*(A\times \mathcal{M}(w)\times A\times \mathcal{M}(w))$ via the map
$$
c: x_1\otimes y_1\otimes x_2\otimes y_2\mapsto -(x_1, x_2)(y_1\otimes y_2),
$$
where $(\cdot, \cdot)$ is the Mukai pairing (\ref{Generalized Mukai pairing}). By \cite[Theorem 1.2(2)]{Markman1}, for each $g\in \mathrm{Spin}(V)_w$, there exists a topological complex line bundle $L_g$ such that
$$
(g\otimes \mathrm{mon}_g)(ch(\mathcal{E})) = ch(\mathcal{E})\pi^*_{\mathcal{M}(w)}ch(L_g),
$$
where $\pi_{\mathcal{M}(w)}$ is the projection map $A\times \mathcal{M}(w)\rightarrow \mathcal{M}(w)$. For $g\in \mathrm{Spin}(V)_w$ an isometry of the Mukai lattice, we may identify
$$
c\circ [(g\otimes \mathrm{mon}_g)\otimes (g\otimes \mathrm{mon}_g)] = c\circ [(1\otimes \mathrm{mon}_g)\otimes (1\otimes \mathrm{mon}_g)]. 
$$
Hence, we obtain 
\begin{equation}\label{chern character equation}
(\mathrm{mon}_g\otimes \mathrm{mon}_g)(ch([E^{\bullet}])) = ch([E^{\bullet}])p_1^*ch(-L_g)p_2^*ch(L_g),
\end{equation}
for $p_i$ the projection maps from $\mathcal{M}(w)\times \mathcal{M}(w)$. We therefore conclude that the characteristic classes $\kappa_i([E^{\bullet}])$ are invariant under the diagonal monodromy action.

 Denote $e: = [ch(\mathcal{E})]_*$. Then for $x\in H^*(A)$,
$$
e(g(x)) = \mathrm{mon}_g(e_x)\pi^*_{\mathcal{M}(w)}ch(L_g).
$$
In particular, 
$$
e_w = \mathrm{mon}_g(e_w)\pi^*_{\mathcal{M}(w)}ch(L_g).
$$
Taking the summand in $H^2(\mathcal{M}(w))$, we get 
$$
c_1(e_w) = c_1(\mathrm{mon}_g(e_w))+(w, w)c_1(L_g)
$$
and so we may write the summand of (\ref{chern character equation}) in $H^2(\mathcal{M}(w))$ as 
$$
(\mathrm{mon}_g\otimes \mathrm{mon}_g)(c_1([E^{\bullet}])) =c_1([E^{\bullet}])-p^*_1[c_1(\mathrm{mon}_g(e_w))-c_1(e_w)]+p^*_2[c_1(\mathrm{mon}_g(e_w))-c_1(e_w)].
$$
It follows that the class $c_1([E^{\bullet}])-p_1^*(c_1(e_w))+p_2^*(c_1(e_w))$ is $\mathrm{Spin(V)}_w$-invariant.  By the previous lemma, the $\mathrm{Spin}(V)_w$-invariant subspace of $H^2(\mathcal{M}(w))$ vanishes, hence, 
$$
c_1([E^{\bullet}])-p_1^*(c_1(e_w))+p_2^*(c_1(e_w))\in H^1(\mathcal{M}(w))\otimes H^1(\mathcal{M}(w)).
$$
We may thus write $\delta = c_1([E^{\bullet}])-p_1^*(c_1(e_w))+p_2^*(c_1(e_w))$, for $\delta$ in Remark \ref{decomposition of first chern class}.

Since the classes $\kappa_i([E^{\bullet}])$ are invariant under the diagonal monodromy action, it suffices observe 

$$\kappa_i([E^{\bullet}]) = \theta_i([E^{\bullet}])\left(\mathrm{exp}\left(\frac{-\delta}{2n+2}\right)\right),
$$ 
for each $i$, as we mentioned in Remark \ref{description of the kappa class}. By the $\mathrm{Spin}(V)_w$-invariance of the class $\delta$, it follows that $\theta_i([E^{\bullet}])$ is $\mathrm{Spin}(V)_w$-invariant.

Lastly, let us relate $\theta([E^{\bullet}])$ to $\theta([E^1])$. Since $[E^{\bullet}] = [E^2]-[E^1]$, we have 
$$
\theta([E^{\bullet}])-\theta([E^1]) = ch([\mathcal{O}_{\Delta}])\mathrm{exp}\left(\frac{-p_1^*\alpha-p_2^*\beta}{2n+2}\right).
$$
Recall that $\alpha$ and $\beta$ are the classes described in Remark \ref{decomposition of first chern class} and we have $\alpha = c_1(e_w)$ and $\beta = -c_1(e_w)$.
Let $d: \mathcal{M}(w)\rightarrow \mathcal{M}(w)\times \mathcal{M}(w)$ be the diagonal embedding. We have $ch(\mathcal{O}_{\Delta}) = ch(d_*\mathcal{O}_{\mathcal{M}})$. By the Grothendieck-Riemann-Roch theorem, 
$$
ch(d_*\mathcal{O}_{\mathcal{M}}) = d_*(ch(\mathcal{O}_{\mathcal{M}})td_d).
$$
By the projection formula, the difference between the two classes is 
$$
ch(\mathcal{O}_{\Delta})\mathrm{exp}\left(\frac{-p_1^*\alpha-p_2^*\beta}{2n+2}\right) = d_*\left(d^*\mathrm{exp}\left(\frac{-p_1^*\alpha-p_2^*\beta}{2n+2}\right)td_d\right).
$$
 The right-hand side is 

$$d_*(td_d),$$ which is monodromy invariant. 
\end{proof}

We make essential use of Proposition \ref{monodromy invariance of the theta class} to show that the classes $\theta_i$ remain algebraic under projective deformations of the product $\mathcal{M}_H(v)\times \mathcal{M}_H(v)$. Verbitsky's theory of hyperholomorphic sheaves is essential for this result. First, we must review Markman's description of a universal family of moduli spaces of sheaves on an Abelian surface.

\vspace{5mm}
\subsection{A universal deformation of a moduli space of sheaves}\label{universal deformation}

  We now describe the construction of a universal family of the moduli space of sheaves on an Abelian surface following \cite[Sections 12.1 and 12.5]{Markman1} and \cite{Markman6}. Recall the conventions made in the introduction for $v$ and $\mathcal{M}(v)$. Let $v^{\perp}$ denote the sublattice of $H^{ev}(A, \mathbb{Z})$ orthogonal to $v$ under the BBF pairing. Again, there exists an isometry
$$
\xi: v^{\perp}\longrightarrow H^2(Kum_a(v), \mathbb{Z})
$$
with respect to the Mukai pairing and the pairing (\ref{Generalized Mukai pairing}) by \cite[Theorem 0.2]{Yoshioka}, where $\xi$ is the composition of the homomorphism $v^{\perp}\rightarrow H^2(\mathcal{M}(v), \mathbb{Z})$ with the restriction of $H^2(\mathcal{M}(v), \mathbb{Z})$ to $H^2(Kum_a(v), \mathbb{Z})$.

The five-dimensional period domain of an IHSM $Y$ deformation equivalent to $Kum_n(A)$ is $$
\Omega_{v^{\perp}}: = \{\ell\in \mathbb{P}(v^{\perp}\otimes_{\mathbb{Z}}\mathbb{C}): (\ell, \ell) = 0, (\ell, \overline{\ell})>0\}.
$$ 
Note that $\Omega_{v^{\perp}}$ is an open analytic subset of the quadric of isotropic lines in $w^{\perp}_{\mathbb{C}}\subset S^{+}_{\mathbb{C}}: = S^+\otimes_{\mathbb{Z}} \mathbb{C}$.
\begin{Def}\label{marking}
Let $Y$ be an IHSM and fix an isometry $H^2(Y, \mathbb{Q})\rightarrow \Lambda$, for $\Lambda$ a fixed lattice, with respect to the BBF and Mukai pairings. Such an isometry is called a \textit{marking}, and $\mathfrak{M}_{\Lambda}$ denotes the moduli space of $(\Lambda-)$ marked IHSMs. 
\end{Def}

\begin{Rem}\label{remark on moduli space}
There exists a moduli space $\mathfrak{M}_{v^{\perp}}$ of marked IHSMs and a surjective and generically injective period map 
$$
Per: \mathfrak{M}^0_{v^{\perp}}\longrightarrow \Omega_{v^{\perp}}
$$
for each connected component $\mathfrak{M}^0_{v^{\perp}}$ of $\mathfrak{M}_{v^{\perp}}$, by results of \cite{Hu3}. For $\eta: H^2(Y, \mathbb{Z})\rightarrow v^{\perp}$ a fixed isometry, the period map sends the pair $(Y, \eta)$ to the line $\eta(H^{2,0}(Y))$.
There exists a universal family 
\begin{equation}\label{Kummer universal family}
p:\mathcal{Y}\longrightarrow \mathfrak{M}_{v^{\perp}}^0
\end{equation}
of IHSMs of generalized Kummer type, along with a trivialization $\eta: R^2p_*\mathbb{Z}\rightarrow \underline{v}^{\perp}$ with value $\eta_0$ at $t$, by the main result of \cite{Markman6}. 
\end{Rem}

Choose a connected component $\mathfrak{M}_{v^{\perp}}^0$ which contains the generalized Kummer $Kum_n(A)$. Let $t\in \mathfrak{M}_{v^{\perp}}^0$ denote the isomorphism class $(Kum_n(A), \eta)$ for $\eta$ a marking and let $\underline{v}^{\perp}$ be the trivial local system over $\mathfrak{M}_{v^{\perp}}^0$ with fiber $v^{\perp}$. 
Now let 
\begin{equation}\label{albanese map}
alb: \mathcal{M}(v)\rightarrow \mathrm{Alb}^1(\mathcal{M}(v))
\end{equation}
be the Albanese morphism to the Albanese variety of degree $1$. The Abelian fourfold $T: = A\times \mathrm{Pic}^0(A)$ acts on $\mathcal{M}(v)$. Let $\lambda: T\times \mathcal{M}(v)\rightarrow \mathcal{M}(v)$ be the action morphism. The \textit{anti-diagonal} action of $a\in T$ maps \begin{equation}\label{antidiagonal def}(b, F) \mapsto (b-a, \lambda(a, F)).\end{equation} For a point $a\in \mathrm{Alb}^1(\mathcal{M}(v))$, and $v$ a primitive and effective Mukai vector such that $(v, v)\geq 6$, let $Kum_a(v)$ denote the fiber of (\ref{albanese map}) over $a$. Then $Kum_a(v)$ is deformation equivalent to $Kum_n(A)$, where $n = \frac{(v, v)-2}{2}$, by \cite[Theorem 0.2]{Yoshioka}.

\begin{Lem}(\cite[Lemma 10.1(3)]{Markman1})\label{Markman 10.1(3)}
The moduli space $\mathcal{M}(v)$ is isomorphic to the quotient of $T\times Kum_a(v)$ by the anti-diagonal action of $\Gamma_v$.
\end{Lem}
\begin{proof}
    The lemma is known in the case of $\mathcal{M}(w)$. It follows for $\mathcal{M}(v)$ as the action of $A\times \mathrm{Pic}^0(A)$ on $\mathcal{M}(v)$ conjugates to an action of a deformation $A^{\prime}\times \mathrm{Pic}^0(A^{\prime})$ on $\mathcal{M}(w)$ by the main result of \cite{Yoshioka} (cf. \cite[Theorem 9.3]{Markman1}).
\end{proof}

\begin{Rem}\label{isomorphism class of deformation}
If $t_0$ represents the isomorphism class $(Kum_a(v), \eta^{\prime})$, for

$$\eta^{\prime}: H^2(Kum_a(v), \mathbb{Z})\rightarrow v^{\perp}$$
a marking defined by $\eta^{\prime}: = \eta\circ h^{-1}$, and where $h$ is a parallel transport operator, then $t_0\in \mathfrak{M}^0_{v^{\perp}}$, i.e. $t$ and $t_0$ lie in the same connected component of the moduli space.
\end{Rem}

Now let us retain the notation of subsection \ref{monodromy} and let $S^+_{\mathbb{C}}: = S^{+}\otimes_{\mathbb{Z}} \mathbb{C}$ and define $S^-_{\mathbb{C}}$ and $V_{\mathbb{C}}$ analogously. From the Clifford product there arises a homomorphism
$$
S^+_{\mathbb{C}}\otimes V_{\mathbb{C}}\longrightarrow S^-_{\mathbb{C}}.
$$
Let $\ell$ be an isotropic line in $S^{+}_{\mathbb{C}}$ and let $\ell\otimes V_{\mathbb{C}}$ be the restriction of the above homomorphism. Its kernel is $\ell\otimes Z_{\ell}$, where $Z_{\ell}$ is a maximal isotropic subspace of $V_{\mathbb{C}}$. The image of $\ell\otimes V_{\mathbb{C}}$ is a maximal isotropic subspace of $S^-_{\mathbb{C}}$ \cite[pg. 136, III.1.4]{Chevalley}. Furthermore there is a homomorphism $S^+\rightarrow \mathrm{Hom}(Z_{\ell}, S^-_{\mathbb{C}})$ induced by Clifford multiplication. Let $Q(S_{\mathbb{C}}^+)\subset \mathbb{P}(S_{\mathbb{C}}^+)$ be the quadric of isotropic lines and $IG(4, V_{\mathbb{C}})$ the Grassmanian of maximal isotropic subspaces of $V_{\mathbb{C}}$. The map $\ell\mapsto Z_{\ell}$ is an isomorphism between $Q(S^+_{\mathbb{C}})$ and a connected component $IG^+(4, V_{\mathbb{C}})$ of $IG(4, V_{\mathbb{C}})$. Similarly, there is an isomorphism between $Q(S^-_{\mathbb{C}})$ and $IG^-(4, V_{\mathbb{C}})$, where $IG^-(4, V_{\mathbb{C}})$ is the other connected component of $IG(4, V_{\mathbb{C}})$ \cite[pg. 137, III.1.6]{Chevalley}. Define
$$
\Omega_{S^+}: = \{\ell\in \mathbb{P}(S^+_{\mathbb{C}}) : (\ell, \ell) = 0, (\ell, \overline{\ell})>0\},
$$
where the pairing is given by minus the pairing (\ref{Generalized Mukai pairing}). Sending $\ell$ to $Z_{\ell}$, we get a $\mathrm{Spin}(S_{\mathbb{R}}^+)$-equivariant embedding 
$$
\zeta: \Omega_{S^+}\longrightarrow IG^+(4, V_{\mathbb{C}})
$$
of $\Omega_{S^+}$ as an open analytic subset of $IG^+(4, V_{\mathbb{C}})$. The subspace $Z_{\overline{\ell}}$ is the complex conjugate of $Z_{\ell}$. Let $J_{\ell}: V_{\mathbb{C}}\rightarrow V_{\mathbb{C}}$ be the endomorphism acting on $Z_{\ell}$ by $i$ and on $Z_{\overline{\ell}}$ by $-i$. $J_{\ell}$ leaves $V_{\mathbb{R}}$ invariant and hence $J_{\ell}$ induces a complex structure on $V_{\mathbb{R}}$. The choice of $\ell\in \Omega_{S^+}$ endows $S^+$ with an integral weight 2 Hodge structure, such that $(S^+_{\mathbb{C}})^{2,0} = \ell$, and it endows $V$ with an integral weight 1 Hodge structure, such that $V^{1, 0} = Z_{\ell}$. For $(v, v)>0$, we can consider the following alternative definition of $\Omega_{v^{\perp}}$:
$$
\Omega_{v^{\perp}}: = \{\ell\in \Omega_{S^+}: (\ell, v) = 0\},
$$
i.e. as a hyperplane section of $\Omega_{S^+}$.
Now consider the restriction of the period map from $\Omega_{S^+}$ to $\Omega_{v^{\perp}}$. The upshot is the following:
\begin{Prop}(\cite[Proposition 1.6(2)]{Markman1})\label{period domain weight 1 Hodge structures}
In addition to being the period domain of marked generalized Kummers, $\Omega_{v^{\perp}}$ is the period domain of integral weight 1 Hodge structures $(V, J_{\ell})$.
\end{Prop}

It follows that there exists a universal torus $\mathcal{T}$ and a period map \begin{equation}\label{universal torus}\pi: \mathcal{T}\rightarrow \Omega_{v^{\perp}},\end{equation} a fiber of which is a four dimensional complex torus deformation equivalent to $A\times \mathrm{Pic}^0(A)$. Let $Per^*\mathcal{T}\rightarrow \mathfrak{M}_{v^{\perp}}^0$ be the pullback by the period map of the universal torus $\pi: \mathcal{T}\rightarrow \Omega_{v^{\perp}}$.  

\begin{Rem}
The family $\mathcal{T}$ is natural; by \cite[Lemma 12.5]{Markman1} there is a global isogeny between the family of third intermediate Jacobians of $p: \mathcal{Y}\rightarrow \mathfrak{M}_{v^{\perp}}^0$ and the family $Per^*\mathcal{T}$ over $\mathfrak{M}_{v^{\perp}}^0$.
\end{Rem}

Let $\Gamma_v$ be the group given in (\ref{Gamma_v}) isomorphic to $(\mathbb{Z}/(n+1)\mathbb{Z})^4$ and let $\underline{\Gamma_v}$ be the local system associated to the group scheme $\Gamma_v$. In addition, there is a trivial local system $\mathrm{Aut}_0(p)$ of groups of automorphisms of fibers of $p$ acting trivially on the second cohomology. Therefore, the subsystem $\mathcal{Z}$ of $\mathrm{Aut}_0(p)$ acting trivially on the third cohomology also forms a trivial local system. We can thus extend the isomorphism in the diagram of Remark \ref{isomorphism of gamma and torsion points} of the fiber of $\mathcal{Z}$ over $t$ with the group $\Gamma_v$ to a trivialization $\mathcal{Z}\rightarrow \underline{\Gamma_v}$.

\begin{Rem}
A fiber of $\pi$ given by (\ref{universal torus}) is the compact torus $V_{\mathbb{R}}/V$ and hence the local system $\underline{\Gamma_v}$ embeds as a subsystem of the torsion subgroups of $Per^*\mathcal{T}$.
\end{Rem}

Let 
\begin{equation}\label{definition of family of moduli spaces}
\tilde{\mathcal{M}}: = Per^*\mathcal{T}\times_{\underline{\Gamma_v}}\mathcal{Y}
\end{equation}
be the quotient of the fiber product over $\mathfrak{M}_{v^{\perp}}^0$ by the anti-diagonal action of $\underline{\Gamma_v}$ (see (\ref{antidiagonal def})). Denote by 
\begin{equation}\label{universal family}
\Pi: \tilde{\mathcal{M}}\longrightarrow \mathfrak{M}_{v^{\perp}}^0,
\end{equation}
the universal family of the moduli space, where $\mathcal{M}_{t_0}$ denotes the fiber of $\Pi$ over $t_0$, for which $t_0$ represents the isomorphism class $(Kum_a(v), \eta^{\prime})$. The fiber $\mathcal{M}_{t_0}$ is naturally isomorphic to $\mathcal{M}(v)$, by \cite[Lemma 10.1]{Markman1}.

Recall that $E^1$ was the sheaf defined in (\ref{RelativeExtension}). Let $E^1_{F}$ denote the restriction of $E^1$ to $\{[F]\}\times \mathcal{M}(v)$ for $[F]\in \mathcal{M}(v)$.
\begin{Prop}(\cite[Theorem 13.3]{Markman1})
 The sheaf $E^1_F$ deforms with $\mathcal{M}(v)$ to a reflexive sheaf, locally free on the complement of a point over every fiber of the universal family $\Pi$.
\end{Prop}

Let us now introduce the concept of a hyperholomorphic sheaf. There exists a sheaf $[F]\in \mathcal{M}(v)$ such that $E_F^1$ is maximally twisted, i.e. its class in the analytic Brauer group $H^2_{an}(\mathcal{M}(v), \mathcal{O}^*_{\mathcal{M}(v)})$ is of order equal to the rank $(v, v)$ of $E_F^1$. As a result, $E_F^1$ has no proper non-trivial subsheaves and is therefore slope-stable with respect to every K\"{a}hler class on $\mathcal{M}(v)$. 
\begin{Prop}(\cite[Theorem 11.1]{Markman1})\label{Markman Theorem 11.1}
 The class $c_2(\mathcal{E}nd(E_F^1))\in H^4(\mathcal{M}(v), \mathbb{Z})$ is $\mathrm{Spin}(V_{\mathbb{Z}})_v$-invariant.
\end{Prop}

If $c_2(\mathcal{E}nd(E_F^1))$ is $\mathrm{Spin}(V_{\mathbb{Z}})_v$-invariant, it remains of Hodge-type $(2, 2)$ over $\mathfrak{M}_{v^{\perp}}^0$, by \cite[Lemma 12.14]{Markman1}. Any sheaf $\mathcal{F}$ that satisfies the properties that 1) it is slope-stable and 2) $c_2(\mathcal{E}nd(\mathcal{F}))$ is monodromy-invariant, is called \textit{hyperholomorphic} and such a sheaf deforms to a twisted sheaf over each fiber of the universal family $\Pi$ (see \cite{Verbitsky1}).

Given an analytic subset $C$ of $\mathfrak{M}_{v^\perp}^0$, let $\tilde{\mathcal{M}}\times_C\tilde{\mathcal{M}}$ denote the Cartesian product of Markman's universal family of moduli spaces over $C$. At its core, the following proposition supplies the algebraic cycles on projective deformations of moduli spaces of sheaves:

\begin{Prop}(cf. \cite[Theorem 13.3]{Markman1})\label{twistor family}
Let $t_1$ be a point in $\mathfrak{M}^{0}_{v^{\perp}}$ parametrizing the moduli space $\mathcal{M}(v)$. For every point $t_2\in \mathfrak{M}_{v^{\perp}}^0$, there exists a connected reduced projective curve $C$ in $\mathfrak{M}_{v^{\perp}}^0$ containing $t_1$ and $t_2$ and  a flat reflexive sheaf $\mathcal{F}$ over $\tilde{\mathcal{M}}\times_C \tilde{\mathcal{M}}$, locally free away from the diagonal, which restricts to $E^1$ over $t_1$.
\end{Prop}

\begin{Rem}\label{significance of the theta class}
    We can now explain the purpose of the definition of the class $\theta([E^1])$ in (\ref{theta class 2}). As Proposition \ref{twistor family} indicates, a deformation of the sheaf $E^1$ restricts over $t_2\in C$ to a twisted coherent sheaf \begin{equation}\label{twisted sheaf}E^1_{X}: = \mathcal{F}|_{X\times X},\end{equation} where $X\times X$ is the fiber over $t_2\in C$. Therefore, there is ambiguity when defining the classes $ch_i([E^1_{X}])$. However, this ambiguity disappears when we consider the classes $\theta_i([E^1_{X}])$. We have verified that the class $\theta_i([E^1])$ is monodromy invariant (Proposition \ref{monodromy invariance of the theta class}), and we can now verify the algebraicity of $\theta_i([E^1_X])$ under projective deformations of the moduli space:
\end{Rem}

\begin{Prop}\label{algebraicity of the theta class}
 Let $E^1_{X}$ denote the twisted coherent sheaf (\ref{twisted sheaf}) over the fiber $X\times X$ over $t_2$ arising from 
a projective deformation of $\mathcal{M}(w)\times \mathcal{M}(w)$. The characteristic class $\theta_i([E^1_{X}])$ is algebraic for each $i$.
\end{Prop}

\begin{proof}
    The class $\kappa_i([E^1_X])$ is algebraic for each $i$; see Remark \ref{kappa class is analytic} and the first part of the proof of Proposition \ref{monodromy invariance of the theta class}. The classes $\kappa_i$ and $\theta_i$ are related by 
    $$
    \kappa_i([E^1_X]) = \theta_i([E^1_X])\mathrm{exp}\left(\frac{-\delta^{\prime}}{2n+2}\right), 
    $$
    where $\delta^{\prime}$ is of Hodge-type $(1, 1)$ by the proof of Proposition \ref{monodromy invariance of the theta class}. Hence, $\theta_i([E^1_X])$ is obtained by taking the cup product of algebraic classes. The result then follows as a corollary of Proposition \ref{monodromy invariance of the theta class}.
\end{proof}

\textbf{\section{An algebraic correspondence on deformations of the moduli space}\label{correspondence}}

  Recall the convention that $w: = w_{n+1}$ is the Chern character of an ideal sheaf of a length-$(n+1)$ subscheme on $A$. The objective of this section is to prove the following:

\begin{Th}\label{LSCforM}
 Let $\mathfrak{M}_{w^{\perp}}^{0}$ be a connected component of $\mathfrak{M}_{w^{\perp}}$ and $\Pi: \tilde{\mathcal{M}}\rightarrow \mathfrak{M}^0_{w^{\perp}}$ be the universal family (\ref{universal family}), one fiber of which is $\mathcal{M}(w)$. Then the Lefschetz standard conjecture holds for any projective fiber of $\Pi$.
\end{Th}

This is a generalization of the result of \cite[Corollary 7.9.]{Ar}, which states that the Lefschetz standard conjecture holds for the moduli space $\mathcal{M}_H(v)$, where $v$ is the Chern character of a Gieseker-stable sheaf on $A$ and $H$ is a $v$-generic polarization. This is indeed a generalization; the generic fiber of the family $\Pi$ is not a moduli space of sheaves, as moduli spaces appear over a three dimensional family of polarized Abelian surfaces, while projective fibers of $\Pi$ appear over a four dimensional moduli space of polarized IHSMs of generalized Kummer type. The proof of Theorem \ref{LSCforM} boils down to a reworking of the main result of \cite{CM} for the case of moduli spaces of sheaves on an Abelian variety and their projective deformations.

\subsection{A self-adjoint algebraic correspondence}\label{self adjoint correspondence}

  We describe correspondences in $H^*(\mathcal{M}(w)\times \mathcal{M}(w), \mathbb{Q})$, which, after a normalization, satisfy the assumptions of Corollary \ref{CharlesCorollary}. In subsection \ref{LSCM} we relate these correspondences to the algebraic correspondences arising from the $\theta$-classes. These results are in turn applied to prove Theorem \ref{LSCforM}, as we have established that the correspondences deform in the universal family $\Pi$. A key property of these correspondences is that they are \textit{self-adjoint}. We recall the basic definition of adjoint operators on vector spaces: 
\begin{Def}
 Let $W_1$ and $W_2$ be vector spaces over $\mathbb{Q}$ with associated symmetric bilinear forms $(\cdot, \cdot)_{W_i}$. Let $f: W_1\rightarrow W_2$ be a homomorphism. Then the \textit{adjoint of $f$}, which we denote by $f^{\dagger}$, is defined by the property
 $$
 (x, f^{\dagger}(y))_{W_1}: = (f(x), y)_{W_2}.
 $$
 
\noindent We say $f$ is \textit{self-adjoint} if $W_1 = W_2$ and $f = f^{\dagger}$.
\end{Def}

Let $w\in \Lambda(A)$ be the element of the Mukai lattice representing the ideal sheaf of a length $(n+1)$-subscheme on $A$. Define the maps

\begin{equation}
\begin{split}
& h_w: K^{\bullet}(A)\rightarrow \mathbb{Q}w,\\
& h_{w^{\perp}}: K^{\bullet}(A)\rightarrow w_{\mathbb{Q}}^{\perp},\\
& h_{1}: K^{\bullet}(A)\rightarrow K^1(A)
\end{split}
\end{equation}

\noindent to be orthogonal projections with respect to the generalized Mukai pairing on $K^{\bullet}(A)$ (see Remark \ref{Generalized Mukai pairing}).

Let $\pi_i$ be the projection of $A\times \mathcal{M}(w)$ onto the $i$th factor. Let $\mathcal{E}$ be a universal sheaf on $A\times \mathcal{M}(w)$ and let $[\mathcal{E}]$ denote the corresponding class in $K^{\bullet}(A\times \mathcal{M}(w))$. Define $\phi^{\prime}: K^{\bullet}(A)\rightarrow K^{\bullet}(\mathcal{M}(w))$ via 

\begin{equation}\label{phi prime}
\phi^{\prime}(x): = \pi_{2_!}(\pi_1^{!}(x)\otimes [\mathcal{E}]).
\end{equation}

Now let $\psi^{\prime}: K^{\bullet}(\mathcal{M}(w))\rightarrow K^{\bullet}(A)$ be the right adjoint of $\phi^{\prime}$ with respect to the Mukai pairings on $K^{\bullet}(A)$ and $K^{\bullet}(\mathcal{M}(w))$. It is a general fact (see for example \cite[Def. 5.7 and Prop. 5.9]{Hu2}) that given a Fourier-Mukai functor $\Phi: D^b(A)\rightarrow D^b(\mathcal{M}(w))$ with Fourier-Mukai kernel $\mathcal{F}$, the Fourier-Mukai functor $\Psi: D^b(\mathcal{M}(w))\rightarrow D^b(A)$ with Fourier-Mukai kernel $\mathcal{F}^{\vee}\otimes \pi_1^*\omega_A[2]$ is the right adjoint of $\Phi$. Since the dualizing object in $D^b(A)$ is trivial, $\mathcal{F}^{\vee}[2] \cong \mathcal{F}^{\vee}\otimes \pi_1^*\omega_A[2]$ in $D^b(A\times \mathcal{M}(w))$. The $K$-theoretic analogue of this statement is that for the class of a universal sheaf $[\mathcal{E}]\in K^{\bullet}(A\times \mathcal{M}(w))$, the classes $[\mathcal{E}^{\vee}]$ and $[\mathcal{E}^{\vee}\otimes \pi_1^*\omega_A[2]]$ coincide. Therefore, we may describe $\psi^{\prime}$ via

\begin{equation}\label{psi prime}
\psi^{\prime}(y): = \pi_{1_{!}}(\pi_2^{!}(y)\otimes [\mathcal{E}^{\vee}]).
\end{equation}

 Consider next the correspondence 
 \begin{equation}\label{f tilde}
 \tilde{f}^{\prime}: = \phi^{\prime}\circ \psi^{\prime}.
 \end{equation}
 The following is a theorem of Mukai (cf. \cite[Proposition 5.10]{Hu2}): 
 \begin{Prop}
 Let $\Phi_{\mathcal{P}}: D^b(X)\rightarrow D^b(Y)$ and $\Phi_{\mathcal{Q}}: D^b(Y)\rightarrow D^b(Z)$ be Fourier-Mukai transforms with Fourier-Mukai kernels $\mathcal{P}$ and $\mathcal{Q}$ respectively. Then the composition $\Phi_{\mathcal{Q}}\circ \Phi_{\mathcal{P}}: D^b(X)\rightarrow D^b(Z)$ is isomorphic to the Fourier-Mukai transform $\Phi_{\mathcal{R}}$, where $\mathcal{R}\in D^b(X\times Z)$ is the Fourier-Mukai kernel
 $$
 \mathcal{R}: = \pi_{13, *}(\pi^*_{12}\mathcal{P}\otimes \pi^*_{23}\mathcal{Q}).
 $$
 \end{Prop}
 
 This implies the following:
 
 \begin{Cor}
  Let $p_i$ denote the projection of $\mathcal{M}(w)\times \mathcal{M}(w)$ onto the $i$th factor. The map $\tilde{f}^{\prime}: K^{\bullet}(\mathcal{M}(w))\rightarrow K^{\bullet}(\mathcal{M}(w))$ given by (\ref{f tilde}) may be described via

\begin{equation}
\tilde{f}^{\prime}(x): = p_{2_!}(p_1^{!}(x)\otimes [E^{\bullet}]),
\end{equation}
where $[E^{\bullet}]$ is the class given by (\ref{KclassRelativeExtension}).
\end{Cor}

\begin{Cor}
The correspondence $\tilde{f}^{\prime}$ given by (\ref{f tilde}) is self-adjoint.
\end{Cor}

 Conjugate $\tilde{f}^{\prime}$ by the map $\tilde{ch}: = ch\sqrt{td_{\mathcal{M}}}: K^{\bullet}(\mathcal{M}(w))\rightarrow H^*(\mathcal{M}(w))$ defined by $[F]\mapsto ch([F])\sqrt{td_{\mathcal{M}}}$ to get $f^{\prime}: = \tilde{ch}\circ \tilde{f}^{\prime}\circ (\tilde{ch})^{\dagger}: H^*(\mathcal{M}(w))\rightarrow H^*(\mathcal{M}(w))$. Note that $\Tilde{ch}$ is an isometry and hence $(\Tilde{ch})^{-1} = (\Tilde{ch})^{\dagger}$. The correspondence $f^{\prime}$ is the self-adjoint correspondence given by 
\begin{equation}\label{AlgebraicCorrespondence}
   f^{\prime}: =  ch([E^{\bullet}])\sqrt{td_{\mathcal{M}(w)\times \mathcal{M}(w)}}.
\end{equation}
Note also that $f^{\prime}$ is an algebraic correspondence. By Remark \ref{decomposition of first chern class}, we write
$$
 c_1([E^{\bullet}]) = -p_1^*\alpha - p_2^*\beta - \delta.
$$
Now let $\zeta \in K^{\bullet}(\mathcal{M}(w))$ be a class which satisfies $ch(\zeta) = \mathrm{exp}(\frac{-\alpha}{2n+2})$ and $\zeta^{\prime}$ be a class that satisfies $ch(\zeta^{\prime}) = \mathrm{exp}(\frac{\beta}{2n+2})$. Let $\tau_{\zeta}: K^{\bullet}(\mathcal{M}(w))\rightarrow K^{\bullet}(\mathcal{M}(w))$ be defined by $\tau_{\zeta}(x) = x\otimes \zeta$. We define the following maps:

\begin{equation}\label{normalization}
 \begin{split}
     & \phi: = \tau_{\zeta}\circ \phi^{\prime},\\
     & \psi: =  \psi^{\prime} \circ \tau_{\zeta^{\prime}}^{-1},\\
     & \tilde{f}: = \phi\circ \psi,\\
     & f: = \tilde{ch}\circ \tilde{f}\circ (\tilde{ch})^{\dagger}.
 \end{split}  
\end{equation}
A complete description of $f$ is then given by the self-adjoint algebraic class
\begin{equation}\label{finalcorrespondence}
    ch([E^{\bullet}])\mathrm{exp}\left(\frac{-p_1^*\alpha-p_2^*\beta}{2n+2}\right)\left(\sqrt{td_{\mathcal{M}\times\mathcal{M}}}\right).
\end{equation}

Let

\begin{equation}
\begin{split}
& \eta_{w}: \mathbb{Q}w\rightarrow H^*(\mathcal{M}(w)),\\
& \eta_{w^{\perp}}: w_{\mathbb{Q}}^{\perp}\rightarrow H^*(\mathcal{M}(w)),\\
& \eta_1: K^1(A)\rightarrow H^*(\mathcal{M}(w))
\end{split}
\end{equation}
be the restrictions of $\eta:= \tilde{ch}\circ \phi$ to $\mathbb{Q}w$, $w_{\mathbb{Q}}^{\perp}$, and $K^1(A)$ respectively. Then we have the equalities 
\begin{equation}{\label{AdjointReps}}
h_w\circ \eta^{\dagger} = (\eta_w)^{\dagger}, \hspace{5mm} h_{w^{\perp}}\circ \eta^{\dagger} = (\eta_{w^{\perp}})^{\dagger}, \hspace{5mm} h_1\circ \eta^{\dagger} = (\eta_1)^{\dagger}.
\end{equation}

Denote by $pr_i: H^{*}(\mathcal{M}(w))\rightarrow H^i(\mathcal{M}(w))$ the projection onto the $i$th factor. Define 
\begin{equation}\label{projection of eta}
\eta_{i, w}: = pr_i\circ \eta_w, \hspace{5mm} \eta_{i, w^{\perp}}: = pr_i\circ \eta_{w^{\perp}}, \hspace{5mm} \eta_{i, 1}: = pr_i\circ \eta_1.
\end{equation}

Denote by $$\iota_{4n+8-i}: H^{4n+8-i}(\mathcal{M}(w))\rightarrow H^*(\mathcal{M}(w))$$ the inclusion. Then the corresponding adjoint maps are given by
\vspace{3mm}
\begin{equation}
\begin{split}
    &(\eta_{i, w})^{\dagger} = h_w\circ \eta^{\dagger}\circ \iota_{4n+8-i}, \\
    &(\eta_{i, w^{\perp}})^{\dagger} = h_{w^{\perp}}\circ \eta^{\dagger}\circ \iota_{4n+8-i},\\
     &(\eta_{i, 1})^{\dagger} = h_{1}\circ \eta^{\dagger}\circ \iota_{4n+8-i}.
     \end{split}
\end{equation}
\vspace{3mm}

 Lastly, we define \begin{equation}\label{f sub i}f_i: = pr_i\circ f\circ \iota_{4n+8-i}\end{equation} and

\begin{equation}\label{eta}
\eta_i: = pr_i\circ \eta.
\end{equation}
We sum up these constructions by noting that there is a commutative diagram (cf. \cite[6.3]{CM}):
\vspace{7mm}

\begin{center}
\begin{tikzcd}
H^{4n+8-i}(\mathcal{M}(w)) \arrow[rr, "f_i"] \arrow[d, "\iota_{4n+8-i}"] \arrow[rddd, "\eta_i^{\dagger}"', bend right=100] &                                                                          & H^i(\mathcal{M}(w))                   \\
H^*(\mathcal{M}(w)) \arrow[d, "(\tilde{ch})^{\dagger}"] \arrow[rr, "f"]                                                                         &                                                                          & H^*(\mathcal{M}(w)) \arrow[u, "pr_i"]  \\
K^{\bullet}(\mathcal{M}(w)) \arrow[rd, "\psi"] \arrow[rr, "\tilde{f}"]                                                &                                                                          & K^{\bullet}(\mathcal{M}(w)) \arrow[u, "\tilde{ch}"] \\
                                                                                                                   & K^{\bullet}(A) \arrow[ru, "\phi"] \arrow[ruuu, "\eta_i"', bend right=100] &                                   
\end{tikzcd}
\end{center}

\subsection{The Lefschetz standard conjecture for projective deformations of moduli spaces of sheaves}\label{LSCM}

Theorem \ref{LSCforM} will be proven at the end of this subsection. The topological arguments that relate the statement of the LSC to the surjectivity of the correspondence on a quotient space of $H^i(\mathcal{M}(w))$ (see Proposition \ref{surjective correspondence}) and which are necessary for the proof of Theorem \ref{LSCforM} are identical to those found in \cite{CM}.

As in the conventions, we let $\overline{H}^i(\mathcal{M}(w))$ denote $H^{i}(\mathcal{M}(w))/[R^i+Alg^i]$, where $R^i$ consists of degree $i$ classes of $\mathcal{M}(w)$ generated by classes of lower degree and $Alg^i$ consists of algebraic classes of $\mathcal{M}(w)$ of degree $i$. 

\begin{Prop}\label{surjectivity}
The composition 
$$
K^{\bullet}(A)\xrightarrow[]{\eta_i} H^i(\mathcal{M}(w))\longrightarrow 
\overline{H}^i(\mathcal{M}(w))
$$
is surjective for all $i\geq 1$.
\end{Prop}

\begin{proof}
The map $\eta_i: K^{\bullet}(A)\rightarrow H^i(\mathcal{M}(w))$ is the composition of the maps $$K^{\bullet}(A)\xrightarrow[]{\phi} K^{\bullet}(\mathcal{M}(w))\xrightarrow[]{ch}H^*(\mathcal{M}(w))\xrightarrow{p_i} H^i(\mathcal{M}(w)).$$ Recall that $\phi(\alpha) = \pi_{1_!}(\pi_{2}^{!}(\alpha)\otimes [\mathcal{E}])$, for $\mathcal{E}$ a universal sheaf over $A\times \mathcal{M}(w)$. Let $\{\alpha_1, ..., \alpha_8\}$ be an orthogonal basis of $K^0(A)$ and $\{\alpha_9, ..., \alpha_{16}\}$ an orthogonal basis for $K^1(A)$ with respect to the Mukai pairing on $K^{\bullet}(A)$ (see (\ref{Generalized Mukai pairing})). We write $[\mathcal{E}] = \sum_{j = 1}^8\alpha_j\otimes \beta_j + \sum_{k = 9}^{16} \alpha_k\otimes \beta_k$, where $\beta_i = \phi(\alpha_i)$. Now $[\mathcal{E}]\in K^{\bullet}(A\times \mathcal{M}(w))$, so by the K\"{u}nneth theorem for $K^{\bullet}(A\times \mathcal{M}(w))$, $\beta_j\in K^0(\mathcal{M}(w))$ and $\beta_k\in K^1(\mathcal{M}(w))$. The K\"{u}nneth factors of the subspaces $ch_i(\phi(K^{\bullet}(A)))$ generate the cohomology ring $H^*(\mathcal{M}(w))$ by \cite[ Corollary 2]{Markman3}. As a result, $c_i(\beta_j)$ generate $H^{ev}(\mathcal{M}(w))$ and $c_{i+{\frac{1}{2}}}(\beta_k)$ generate $H^{odd}(\mathcal{M}(w))$. The result follows.
\end{proof}

\begin{Prop}\label{mapscoincide}
 Let $\eta_i$ be the map defined in (\ref{eta}). The images of the maps $f_i$ and $\eta_i$ coincide, for $f_i$ given by (\ref{f sub i}) and $\eta_i$ in the previous proposition.  
\end{Prop}
\begin{proof}
The proof of this proposition is very similar to the analogous proof in the $K3^{[n]}$ case: \cite[Claim 7.2]{CM}. The map $\eta_{i, w}$ given in (\ref{projection of eta}) is injective if it does not vanish. We claim that the same is true for $\eta_{i, w^{\perp}}$ and $\eta_{i, 1}$. This follows from the $\mathrm{Spin}(V)_w$-equivariance of $\eta_{i, w^{\perp}}$ and $\eta_{i, 1}$ established as a result of \cite[Corollary 8.7]{Markman1}. We have seen already that $S^-_{\mathbb{Q}}$ is an irreducible $\mathrm{Spin}(V)_w$-representation (see subsection \ref{monodromy}). Furthermore, by Lemma \ref{Decomposition of G-representations} there is a direct sum decomposition $K^{\bullet}(A)\cong \mathbb{Q}w\oplus w^{\perp}_{\mathbb{Q}}\oplus K^1(A)$, where each summand is an irreducible $\mathrm{Spin}(V)_w$-representation.
Therefore, the images of $(\eta_{i, 1})^{\dagger}$, $(\eta_{i, w})^{\dagger}$, and $(\eta_{i, w^{\perp}})^{\dagger}$ are equal to $K^1(A)$, $\mathbb{Q}w$, and $w^{\perp}_{\mathbb{Q}}$ respectively if they do not vanish. Then the compositions $\eta_{i, 1}\circ (\eta_{i, 1})^{\dagger}$, $\eta_{i, w}\circ (\eta_{i, w})^{\dagger}$, and $\eta_{i, w^{\perp}}\circ (\eta_{i, w^{\perp}})^{\dagger}$ have the same images as $\eta_{i, 1}$, $\eta_{i, w}$, and $\eta_{i, w^{\perp}}$ respectively. For each $i$, we may then express
$$
\eta_i = \eta_{i, w}\circ h_w + \eta_{i, w^{\perp}}\circ h_{w^{\perp}} + \eta_{i, 1}\circ h_{1}.
$$
Correspondingly, the adjoint is given by $\eta_i^{\dagger} = (\eta_{i, w})^{\dagger} + (\eta_{i, w^{\perp}})^{\dagger} + (\eta_{i, 1})^{\dagger}$. By the construction of $f_i$, we have 
$$
f_i = \eta_{i, 1}\circ (\eta_{i, 1})^{\dagger} + \eta_{i, w}\circ (\eta_{i, w})^{\dagger}+ \eta_{i, w^{\perp}}\circ (\eta_{i, w^{\perp}})^{\dagger},
$$
and so the result follows.
\end{proof}

\begin{Prop}\label{surjective correspondence}
For $f_i$ the map given by (\ref{f sub i}), the composition $$H^{4n+8-i}(\mathcal{M}(w))\xrightarrow[]{f_i} H^{i}(\mathcal{M}(w))\rightarrow \overline{H}^i(\mathcal{M}(w))$$ is surjective for $i\geq 1$.
\end{Prop}

\begin{proof}
This follows immediately from Propositions \ref{surjectivity} and \ref{mapscoincide}.
\end{proof}

Let $g_i: H^{4n+8-i}(\mathcal{M}(w))\rightarrow H^{i}(\mathcal{M}(w))$ denote the map induced by the degree $i$ graded summand of the cycle 
\begin{equation}\label{g}
    -\theta([E^1])\sqrt{td_{\mathcal{M}(w)\times \mathcal{M}(w)}}.
\end{equation}
Recall that $\theta$ is the characteristic class (\ref{theta class 2}).

\begin{Cor}\label{f and g}
Let $\overline{f}_i: H^{4n+8-i}(\mathcal{M}(w))\rightarrow \overline{H}^i(\mathcal{M}(w))$ be the composition of $f_i$ (\ref{f sub i}) with the surjection $H^{i}(\mathcal{M}(w))\rightarrow \overline{H}^i(\mathcal{M}(w))$. Define $\overline{g}_i$ analogously, where $g_i$ is the map defined in (\ref{g}). Then $\overline{f}_i = \overline{g}_i$ for $i$ in the range $1< i< 2n+4$. In particular, the correspondence $\overline{g}_i$ is surjective for $1< i< 2n+4$.
\end{Cor}

\begin{proof}
Let $p_i$ denote the projection of $\mathcal{M}(w)\times \mathcal{M}(w)$ onto its factors. Since
$$
ch([E^{\bullet}]) = ch([E^1])-ch([\mathcal{O}_{\Delta}]),
$$
we have
$$
f_i-g_i = ch_i([\mathcal{O}_{\Delta}])\mathrm{exp}\left(\frac{-p_1^*\alpha - p_2^*\beta}{2n+2}\right)\left(\sqrt{td_{\mathcal{M}(w)\times \mathcal{M}(w)}}\right).
$$
The above fact follows since the sheaf cohomology of the complex $E^{\bullet}$ in degree $2$ is given by $\mathcal{E}xt^2_{\pi_{13}}(\pi_{12}^*\mathcal{E}, \pi_{23}^*\mathcal{E})\cong \mathcal{O}_{\Delta}$. 
For $i<2n+4$, we have $ch_i(\mathcal{O}_{\Delta}) = 0$, therefore $f_i = g_i$ in degree $i<2n+4$, and so $\overline{f}_i = \overline{g}_i$ in this range. The surjectivity of $\overline{g}_i$ for $1<i<2n+4$ follows from this fact combined with Proposition \ref{surjective correspondence}.
\end{proof}

In order to prove Theorem \ref{LSCforM} by means of applying Proposition \ref{surjective correspondence} and Corollary \ref{f and g}, we need a series of topological results that will be necessary for the inductive argument found in the proof. These results can be found in \cite{CM}. We repeat the proofs here for clarity.

\begin{Lem}(\cite[Lemma 2.2(1)]{CM})
Let $X$ be a smooth projective variety of complex dimension $d$ and $i\leq d$ an integer. Assume $i = 2j$ and let $\alpha\in H^i(X)$ be the class of a codimension-$j$ algebraic cycle in $X$. Then there exists an algebraic cycle $Z$ of codimension $i$ in $X\times X$ such that the image of the correspondence 
$$
[Z]^*: H^{2d-i}(X)\rightarrow H^i(X)
$$
contains $\alpha$.
\end{Lem}
\begin{proof}
Let $\alpha\in H^i(X)$ be the class of a codimension $j$ algebraic cycle $T$. Then the codimension $i$ cycle $T\times T$ is algebraic in $X\times X$. Let $[Z]$ be the class of $T\times T$ in $H^{2i}(X\times X)$. The image of $[Z]$ in the projection $H^{2i}(X\times X)\rightarrow H^i(X)\otimes H^i(X)$ is $\alpha\otimes \alpha$. Therefore, the image of the correspondence
$$
[Z]^* : H^{2d-i}(X)\rightarrow H^i(X)
$$
is the line generated by $\alpha$.
\end{proof}

\begin{Lem}(\cite[Lemma 2.2(2)]{CM})\label{Lemma2}
Assume that $X$ satisfies the Lefschetz standard conjecture in degrees up to $i-1$. Then $X\times X$ satisfies the Lefschetz standard conjecture in degrees up to $i-1$.
\end{Lem}
\begin{proof}
Assume that $X$ satisfies the Lefschetz standard conjecture in degrees up to $i-1$. We need to show that $X\times X$ satisfies the Lefschetz standard conjecture in degree $i-1$, by induction.

Let $j\in \{0, ..., i-1\}$, then there exists a codimension $j$ algebraic cycle $Z_j$ in $X\times X$ such that the correspondence
$$
[Z_j]^*: H^{2d-j}(X)\longrightarrow H^j(X)
$$
is an isomorphism. Let $k\in \{0, ..., 2d\}$. Let $\pi^k: H^*(X)\rightarrow H^k(X)$ be the projection operator. Then $\pi^k\in H^{2d-k}(X)\otimes H^k(X)\subset H^{2d}(X\times X)$ is the $k$th K\"{u}nneth component of the diagonal. By \cite[ Lemma 2.4]{K1}, the components $\pi^0, ..., \pi^{i-1}, \pi^{2d-i+1}, ... \pi^{2d}$ are algebraic, i.e. the projections
$$
\pi^j: H^*(X)\rightarrow H^j(X)\hookrightarrow H^*(X)
$$
and
$$
\pi^{2d-j}: H^*(X)\rightarrow H^{2d-j}\hookrightarrow H^*(X)
$$
are induced by algebraic correspondences. After possibly replacing the correspondence $[Z_j]^*$ with $[Z_j]^*\circ \pi_{2d-j}$, we can assume the correspondence
$$
[Z_j]^*: H^{2d-k}(X)\rightarrow H^{2j-k}(X)
$$
is zero unless $k = j$.
Consider the codimension $i-1$ cycle $Z$ in $(X\times X)\times (X\times X)$ defined by 
$$
Z: = \sum_{j = 0}^{i-1}Z_j\times Z_{i-1-j}.
$$
We claim that the correspondence 
$$
[Z]^*: H^{4d-i+1}(X\times X)\longrightarrow H^{i-1}(X\times X)
$$
is an isomorphism. For $j\in \{0,..., i-1\}$, the hypothesis on the cycles $Z_j$ imply that the correspondence
$$
[Z_j\times Z_{i-1-j}]^*: H^{4d-i+1}(X\times X)\rightarrow H^{i-1}(X\times X)
$$
maps $H^{2d-k}(X)\otimes H^{2d-i+1+k}(X)\subset H^{4d-i+1}(X\times X)$ to zero unless $k = j$. Therefore, $[Z_j\times Z_{i-1-j}]^*$ maps $H^{2d-j}(X)\otimes H^{2d-i+1+j}(X)$ isomorphically onto $H^j(X)\otimes H^{i-1-j}(X)$. The result follows.
\end{proof}

\begin{Lem}(\cite[Lemma 2.2(2) cont.]{CM})\label{Lemma3}
Suppose $X$ satisfies the Lefschetz standard conjecture in degree up to $i-1$. Let $j$ and $k$ be positive integers with $i = j+k$. Then there exists a cycle $Z$ of codimension $i$ in $(X\times X)\times X$ such that the image of the correspondence
$$
[Z]^*: H^{4d-i}(X\times X)\longrightarrow H^i(X)
$$
contains the image of $H^j(X)\otimes H^k(X)$ in $H^i(X)$ by cup-product.
\end{Lem}
\begin{proof}
Since $j$ (resp. $k$) is less than or equal to $i-1$, $X$ satisfies the Lefschetz standard conjecture in degree $j$, (resp. $k$). Hence, there exists a cycle $T$ (resp. $T^{\prime}$) in $X\times X$ such that 
$$
[T]^*: H^{2d-j}(X)\longrightarrow H^j(X)
$$
(resp. $[T^{\prime}]^*: H^{2d-k}(X)\rightarrow H^k(X)$) is an isomorphism. Let $p_{ij}$ denote the projection of $X\times X\times X$ onto the $i$th and $j$th components. Let $Z\in CH^i(X\times X\times X)$ be the intersection of $p^*_{13}T$ and $p^*_{23}T^{\prime}$. It follows that the image of the correspondence
$$
[Z]^*: H^{4d-i}(X\times X)\longrightarrow H^i(X)
$$
contains the image of $H^j(X)\otimes H^k(X)$ in $H^i(X)$ by the cup product.
\end{proof}

\begin{Prop}({\cite[Proposition 8]{C}}\label{Charles prop 8})
Let $X$ be a smooth projective variety of dimension $d$, and let $i\leq d$ be an integer. Then the Lefschetz standard conjecture is true in degree $i$ for $X$ if and only if there exists a disjoint union $S$ of smooth projective varieties of dimension $l\geq i$ satisfying the Lefschetz standard conjecture in degrees up to $i-2$ and a codimension $i$ cycle $Z$ in $X\times S$ such that the morphism 
$$
[Z]^*: H^{2l-i}(S)\longrightarrow H^i(X)
$$
induced by the correspondence $Z$ is surjective.
\end{Prop}

\begin{Cor}(\cite[Corollary 2.4]{CM}\label{CharlesCorollary})
Let $X$ be a smooth projective variety of complex dimension $d$, and let $i\leq d$ be an integer. Suppose that $X$ satisfies the Lefschetz standard conjecture in degrees up to $i-1$. Let $R^i(X)\subset H^i(X)$ be the subspace of classes which belong to the subring generated by classes of degree $<i$, and let $Alg^i(X)\subset H^i(X)$ be the subspace of $H^i(X)$ generated by the cohomology classes of algebraic cycles.\footnote{Note that $Alg^i(X)$ vanishes for odd $i$.}

If there exists a cycle $Z$ of codimension $i$ in $X\times X$ such that the image of the morphism 
$$
[Z]^*: H^{2d-i}(X)\longrightarrow H^i(X)
$$
maps surjectively onto the quotient space $H^i(X)/[Alg^i(X)+R^i(X)]$, then $X$ satisfies the Lefschetz standard conjecture in degree $i$.
\end{Cor}
\begin{proof}
This follows from Lemmas \ref{Lemma2} and \ref{Lemma3}: Let $\alpha_1, ..., \alpha_r$ be a basis for $Alg^i(X)$, the collection of algebraic classes of $X$ in degree $i$. There exist codimension $i$ cycles $Z_1, ..., Z_r$ in $X\times X$ and cycles $(Z_{j, k})_{j, k>0, j+k = i}$ in $(X\times X)\times X$ for which the image of the correspondence
$$
[Z_{l}]^*: H^{2d-i}(X)\longrightarrow H^i(X)
$$
contains $\alpha_l$ for $1\leq l\leq r$ and such that the image of the correspondence
$$
[Z_{j, k}]^*: H^{4d-i}(X\times X)\longrightarrow H^{i}(X)
$$
contains the image of $H^j(X)\otimes H^k(X)$ in $H^i(X)$. We have that $X\times X$ satisfies the LSC in degree up to $i-1$. The disjoint union of the cycles $Z\times X$, $(Z_l\times X)_{1\leq l\leq r}$ and $(Z_{j, k})_{j, k>0, j+k = i}$ in a disjoint union of $(X\times X)\times X$ thus satisfy the hypothesis of Proposition $\ref{Charles prop 8}$. The space generated by the images of $H^i(X)$ of the correspondences $[Z_{j, k}]^*$ contains $A^i(X)$ by definition. Adding the images in $H^i(X)$ of the correspondences $[Z_l\times X]^*$ which generate a space containing $Alg^i(X)$, and the image in $H^i(X)$ of $[Z\times X]^*$, which maps surjectively onto $H^i(X)/[R^i(X)+Alg^i(X)]$, we get the entire space $H^i(X)$. This shows that $X$ satisfies the Lefschetz standard conjecture in degree $i$.
\end{proof}

We can now prove Theorem \ref{LSCforM}:

\begin{proof}[Proof of Theorem \ref{LSCforM}]
Recall the convention that $w: = w_{n+1}$. By the construction in subsection \ref{universal deformation}, there exists a universal family $\Pi: \tilde{\mathcal{M}}\rightarrow \mathfrak{M}_{w^{\perp}}^0$ given by (\ref{universal family}), one fiber of which is $\Pi^{-1}(t_1) = \mathcal{M}(w)$, and choose a point $t_2\in \mathfrak{M}^0$ parametrizing the fiber $X: = \Pi^{-1}(t_2)$. Choose a reduced connected projective curve $C\subset \mathfrak{M}_{w^{\perp}}^0$ as in Proposition \ref{twistor family}, which contains both $t_1$ and $t_2$. We then consider the natural morphism 
$$
q: \mathcal{\tilde{M}}\times_{C}\mathcal{\tilde{M}}\longrightarrow C.
$$
  Let $\Pi_C: \tilde{\mathcal{M}_C}\rightarrow \mathfrak{M}^0_{w^{\perp}}$ denote the restriction of $\Pi$ to $C$, where $\tilde{\mathcal{M}}_C$ is the restriction of $\tilde{\mathcal{M}}$ to $C$. By Proposition \ref{twistor family}, there is a flat reflexive sheaf $\mathcal{F}$ over $\tilde{\mathcal{M}}\times_C\tilde{\mathcal{M}}$ which restricts to the sheaf $E^1$ over $t_1$. This implies there is a flat section $\sigma$ of the local system $R^*q_*\mathbb{Q}$ through the class
 
 $$-\theta(E^1)\in H^*(\mathcal{M}(w)\times \mathcal{M}(w))$$ 
 which is algebraic in $H^*(X\times X)$. Recall that $\theta$ was the characteristic class defined in (\ref{theta class 2}). This in turn implies that there exists a flat section $\rho$ of the local system $R^*q_*\mathbb{Q}$ through the class $g\in H^*(\mathcal{M}(w)\times \mathcal{M}(w))$, where $g$ is the correspondence given in (\ref{g}), and which is algebraic in $H^*(X\times X)$, by Proposition \ref{algebraicity of the theta class}. 
 
 We now prove the LSC for $X$ by induction on degree $i\leq 2n+4$: For $i = 0$ the result is clear. Let $i\leq 2n+4$ be a positive integer. Assume the Lefschetz standard conjecture holds for $X$ in degrees up to $i-1$. Let $t_2\in C$ be the point such that $\Pi_C^{-1}(t_2) = X$. Then the correspondence 
$$
\rho_i(t_2): H^{4n+8-i}(X)\longrightarrow H^i(X)
$$
is the fiber of a morphism of local systems 
$$
R^{4n+8-i}\Pi_{C,*}\mathbb{Q}\longrightarrow R^{i}\Pi_{C,*}\mathbb{Q}.
$$
By Corollary \ref{f and g}, the morphism
$$
\overline{g_i}: H^{4n+8-i}(\mathcal{M}(w))\longrightarrow \overline{H}^i(\mathcal{M}(w))
$$
    is surjective for $i\geq 2$. Let $\overline{\rho_i}(t_2)$ denote the composition of the fiber
    $\rho_i(t_2)$
    and the surjection $H^i(X)\rightarrow \overline{H}^i(X)$.
    It follows that $\overline{\rho}_i(t_2)$ is also surjective for $i \geq 2$, as $\overline{\rho}_i(t_2)$ is the composition of the correspondence $\overline{g}_i$ with a parallel transport operator. As a result, $X$ satisfies the Lefschetz standard conjecture in degree $i$ by Corollary \ref{CharlesCorollary}. It follows that $X$ satisfies the Lefschetz standard conjecture by induction.
\end{proof}

\textbf{\section{The main results}\label{LCS}}

In this section we prove Theorem \ref{maintheorem}. Implicit in our notation will be the result of Lemma \ref{deformation invariant}, i.e. that the subgroup $$\Gamma_w \subset \mathrm{Aut}_0(Kum_n(A)),$$ given by (\ref{Gamma_v}) deforms as a local system over $\mathfrak{M}_{w^{\perp}}^0$. For a fiber $Y$ of the universal family $p: \mathcal{Y}\rightarrow \mathfrak{M}_{w^{\perp}}^0$, we denote the corresponding fiber of the local local system as \begin{equation}\label{deformed gamma}\Gamma\subset \mathrm{Aut}_0(Y).\end{equation}
Let $$i: Kum_n(A)\hookrightarrow\mathcal{M}(w)$$ be the inclusion. Over $\mathbb{Q}$-coefficients, the image of the pullback of the inclusion is precisely the subring of $H^*(Kum_n(A), \mathbb{Q})$ invariant under the translation action of $\Gamma$. We denote $Im(i^*) = H^*(Kum_n(A), 
\mathbb{Q})^{\Gamma}$. 

\vspace{5mm}
\subsection{The restriction of the correspondence to the generalized Kummer cohomology ring}\label{restriction of the correspondence}

As previously, let $w: = w_{n+1}$ denote the Chern character of the ideal sheaf of a length-$(n+1)$ subscheme on $A$. In subsection \ref{Moduli Space} we defined (see (\ref{KclassRelativeExtension})) a class $[E^{\bullet}] = [E^2]-[E^1]$ in $K^{\bullet}(\mathcal{M}(w)\times \mathcal{M}(w))$ from which there arises a correspondence $\tilde{f}^{\prime} = [E^{\bullet}]_*:K^{\bullet}(\mathcal{M}(w))\rightarrow K^{\bullet}(\mathcal{M}(w))$. Recall that we denote 
$$
f^{\prime} = ch([E^{\bullet}])\sqrt{td_{\mathcal{M}(w)\times \mathcal{M}(w)}}: H^*(\mathcal{M}(w), \mathbb{Q})\longrightarrow H^*(\mathcal{M}(w), \mathbb{Q})
$$
the corresponding map on cohomology. The correspondence $f$ is the normalization of $f^{\prime}$ given in (\ref{normalization}).

Let us first describe the image of the restriction of the correspondence
$$f: H^*(\mathcal{M}(w), \mathbb{Q})\longrightarrow H^*(\mathcal{M}(w), \mathbb{Q})$$ defined in (\ref{finalcorrespondence}) to a correspondence 
\begin{equation}\label{restriction of f}
f_K: H^*(Kum_n(A), \mathbb{Q})\longrightarrow H^*(Kum_n(A), \mathbb{Q}).
\end{equation}
The map $f$ is associated to a convolution of functors of derived categories $\Phi: D^b(A)\rightarrow D^b(\mathcal{M}(w))$ and $\Psi: D^b(\mathcal{M}(w))\rightarrow D^b(A)$. We consider the composition 
$$
\Phi_K\Psi_K: D^b(Kum_n(A))\longrightarrow D^b(Kum_n(A))
$$
where $\Phi_K: = i^*\circ \Phi$ and $\Psi_K: = \Psi\circ i_*$. The map $\Phi_K$ is a Fourier-Mukai transform with Fourier-Mukai kernel $\mathcal{U}: = (id_A\times i)^*\mathcal{E}$, for $\mathcal{E}$ a universal sheaf over $A\times \mathcal{M}(w)$. $\Psi_K$ is the right-adjoint of $\Phi_K$. Denote by $\phi^{\prime}_K$ and $\psi^{\prime}_K$ their $K$-theoretic analogues (cf. the maps (\ref{phi prime}) and (\ref{psi prime})). The composition $\phi^{\prime}_K\circ \psi^{\prime}_K$ of $K$-rings will be denoted by \begin{equation}\label{f tilde sub K}\Tilde{f}^{\prime}_K: K^{\bullet}_{top}(Kum_n(A))\rightarrow K^{\bullet}_{top}(Kum_n(A))\end{equation} and the map
\begin{equation}\label{f sub K}
f^{\prime}_K = ch([E^{\bullet}])|_{K\times K}\sqrt{td_{Kum_n(A)} \times td_{Kum_n(A)}}
\end{equation}
is the corresponding map of cohomology rings. The correspondence $f_K$ of (\ref{restriction of f}) is therefore the normalization of $f^{\prime}_K$ in analogy to the normalization of $f^{\prime}$ in  (\ref{finalcorrespondence}). There is a commutative diagram

\vspace{5mm}
\begin{center}
\adjustbox{scale = 0.85}{
\begin{tikzcd}
                                                                 &                                                                                     & K^{\bullet}(A) \arrow[rd, "\phi"] \arrow[rrd, "\phi_K", bend left] &                                                    &                                                 \\
K^{\bullet}(Kum_n(A)) \arrow[r] \arrow[rru, "\psi_K", bend left] & K^{\bullet}(\mathcal{M}(w)) \arrow[ru, "\psi"] \arrow[rr, "\tilde{f}"]   \arrow[d]           &                                                                    & K^{\bullet}(\mathcal{M}(w)) \arrow[d] \arrow[r]    & K^{\bullet}(Kum_n(A)) \arrow[ldd, "\tilde{ch}"] \\
                                                                 & {H^*(\mathcal{M}(w), \mathbb{Q})} \arrow[d, "i^*"] \arrow[rr, "f"]         &                                                                    & {H^*(\mathcal{M}(w), \mathbb{Q})} \arrow[d, "i^*"] &                                                 \\
                                                                 & {H^*(Kum_n(A), \mathbb{Q})} \arrow[rr, "f_K"] \arrow[luu, "(\tilde{ch})^{\dagger}"] &                                                                    & {H^*(Kum_n(A), \mathbb{Q})} \arrow[d, "pr_i"]      &                                                 \\
                                                                 & {H^{4n-i}(Kum_n(A), \mathbb{Q})} \arrow[u, "\iota_{4n-i}"] \arrow[rr, "{f_{K, i}}"] &                                                                    & {H^i(Kum_n(A), \mathbb{Q})}                        &                                                
\end{tikzcd}
}
\end{center}
\vspace{5mm}

\begin{Rem} Let $\mathcal{E}$ be a universal sheaf over $A\times \mathcal{M}(w)$. By \cite[Theorem 1]{Markman3}, $c_{2n+4}(\mathcal{E})$ is the class of the diagonal of $\mathcal{M}(w)\times \mathcal{M}(w)$. However, $\mathrm{dim}_{\mathbb{C}}(\mathcal{M}(w)) = 2n+4$ while $\mathrm{dim}_{\mathbb{C}}(Kum_n(A)) = 2n$, so the restriction $c_{2n+4}(\mathcal{E})|_{K\times K}$ is not the class of the diagonal of $Kum_n(A)\times Kum_n(A)$. As a result, the image of the correspondence $f_{K}$ does not generate the entire cohomology ring $H^*(Kum_n(A), \mathbb{Q})$.
\end{Rem}

 For $(v, v)\geq 6$, recall that $p: \mathcal{Y}\rightarrow \mathfrak{M}_{v^{\perp}}$ is the universal family (\ref{Kummer universal family}) of $v^{\perp}$-marked generalized Kummers. Let $t_0\in \mathfrak{M}_{v^{\perp}}$ denote a point whose fiber is the isomorphism class $(Kum_a(v), \eta_0)$, where $Kum_a(v)$ denotes a fiber of the Albanese map (\ref{albanese map}), and
    $$
    \eta_0: H^2(Kum_a(v), \mathbb{Q})\longrightarrow v^{\perp}
    $$
    a fixed isometry. Choose a connected component $\mathfrak{M}_{v^{\perp}}^0$ containing $t_0$.  Recall that there is a universal family $\Pi: \Tilde{\mathcal{M}}\rightarrow \mathfrak{M}_{v^{\perp}}^0$ of which the fiber over $t_0$ is $\mathcal{M}(v)$. Let $t$ be a point parametrizing the fiber $Y_{t}$ of $p$ and $\mathcal{M}_{t}$ of $\Pi$. There are corresponding inclusions $i_{K(v)}^*: Kum_a(v)\subset \mathcal{M}(v)$ and $i_{Y_{t}}^*: Y_{t}\subset \mathcal{M}_{t}$. Define the map \begin{equation}\label{phi sub k}\phi^{\prime}_{K}: K^{\bullet}_{top}(A)\rightarrow K^{\bullet}_{top}(Kum_a(v))\end{equation} with respect to the first inclusion.
    
    \begin{Rem}\label{parallel transport}
        By the result \cite[Lemma 2.1]{Markman6}, once we fix a universal family $p: \mathcal{Y}\rightarrow \mathfrak{M}^0_{v^{\perp}}$, every local system over $\mathfrak{M}^0_{v^{\perp}}$ is trivial. It follows that parallel transport over the moduli space $\mathfrak{M}^0_{v^{\perp}}$ is well-defined. Thus, we may define the map \begin{equation}\label{phi prime Y}\phi^{\prime}_{Y_{t}}: K^{\bullet}_{top}(A)\rightarrow K^{\bullet}_{top}(Y_{t})\end{equation} as the composition of $\phi^{\prime}_{K}$ with a parallel transport operator. 
    \end{Rem}

\begin{Prop}\label{phi_K pulls back the Mukai pairing}
The map $\phi^{\prime}_{K}$ given by (\ref{phi sub k}) pulls back the Mukai pairing on $K^{\bullet}(Kum_n(A))$ to $(n+2)$-times the Mukai pairing on $K^{\bullet}(A)$. In particular, the restriction of the Mukai pairing to the image of $\phi^{\prime}_{K}$ is nondegenerate.
\end{Prop}
\begin{proof}
The results of \cite[pg. 1200]{Meachan} state that $\Psi_K\Phi_K$ is the endofunctor on $D^b(A)$ with corresponding Fourier-Mukai kernel $\oplus_{i=0}^{n+1}\mathcal{O}_{\Delta}[-2i]$, where $\Delta$ denotes the diagonal of $A\times A$. The induced morphism $\psi^{\prime}_K\phi^{\prime}_K$ of $K$-rings is associated to the kernel object is therefore $$\sum_{i = 0}^{n+1}[\mathcal{O}_{\Delta}]_* = (n+2)Id_A.$$
The composition $\phi^{\prime}_{K}\psi^{\prime}_{K}$ is hence nondegenerate and in particular, the restriction of the Mukai pairing to the image of $\phi_K^{\prime}$ is nondegenerate. 
\end{proof}

\begin{Cor}\label{surjectivity of the right-adjoint}
The right adjoint $\psi^{\prime}_{K}: K^{\bullet}_{top}(Kum_n(A))\rightarrow K^{\bullet}_{top}(A)$ is surjective. 
\end{Cor}
\begin{proof}
By the previous proposition, there is an adjoint $K$-theoretic morphism $\phi^{\prime}_{K}$ such that the composition $\psi^{\prime}_{K}\phi^{\prime}_{K}$ gives $(n+2)Id_A$ on $K_{top}^{\bullet}(A)$.
\end{proof}

\begin{Cor}\label{image equal}
The image of $\tilde{f}^{\prime}_{K}: K^{\bullet}_{top}(Kum_n(A))\rightarrow K^{\bullet}_{top}(Kum_n(A))$ is equal to the image of $\phi^{\prime}_K$.
\end{Cor}
\begin{proof}
This follows from Proposition \ref{phi_K pulls back the Mukai pairing} and Corollary \ref{surjectivity of the right-adjoint}.
\end{proof}

\begin{Cor}\label{generate translation invariant subring}
The Chern classes of elements in a basis of the image of $f_{Y_t}$ generate the translation invariant subalgebra $H^*(Y_t, \mathbb{Q})^{\Gamma}$.
\end{Cor}
\begin{proof}
Using \cite[Corollary 2]{Markman3}, the Chern classes of a basis of the image of $f$ generate $H^*(\mathcal{M}_t, \mathbb{Q})$. We have by assumption the Chern classes of $f_{Y_t}$ are the classes of the image of the correspondence $f$ restricted to $H^*(Y_t, \mathbb{Q})$. After possibly conjugating the correspondence with a parallel transport operator (see Remark \ref{parallel transport}), we get a correspondence $$f_{Y_t}: H^*(Y_t, \mathbb{Q})\rightarrow H^*(Y_t, \mathbb{Q})$$ which is surjective onto the image of the restriction $i^*: H^*(\mathcal{M}_t, \mathbb{Q})\rightarrow H^*(Y_t, \mathbb{Q})$, by Corollary \ref{image equal}.
\end{proof}

     \begin{Rem}
       For $C$ an analytic subset of $\mathfrak{M}_{w^{\perp}}^0$, denote by $\mathcal{Y}\times_C\mathcal{Y}$ the Cartesian product over $C$ and $\Tilde{p}: \mathcal{Y}\times_C\mathcal{Y}\rightarrow C$ the natural morphism. Denote by $\ell\in \Omega_{v^{\perp}}$ the image of $t$ under the period map and by $T_{\ell}$ the fiber of the map $\pi: \mathcal{T}\rightarrow \Omega_{v^{\perp}}$ over $\ell$ (cf. Proposition \ref{period domain weight 1 Hodge structures}).
       Recall that by definition $\tilde{\mathcal{M}}: = \mathcal{Y}\times_{\underline{\Gamma}_v} Per^*\mathcal{T}$ (see (\ref{definition of family of moduli spaces})). Let $q: \Tilde{\mathcal{M}}\rightarrow Per^*\mathcal{T}/\Gamma_v$ denote the projection. Then the cartesian product $\mathcal{Y}\times_q\mathcal{Y} = (q\times q)^{-1}(\zeta(\mathfrak{M}^0_{v^{\perp}}))$, where $\zeta: \mathfrak{M}^0_{v^{\perp}}\rightarrow Per^*\mathcal{T}/\Gamma_v$ is the zero section. It follows there is a global morphism $\mathcal{Y}\times_q \mathcal{Y}\hookrightarrow \Tilde{\mathcal{M}}\times_q\Tilde{\mathcal{M}}$, and hence a global morphism $\mathcal{Y}\times_C \mathcal{Y}\hookrightarrow \Tilde{\mathcal{M}}\times_C\Tilde{\mathcal{M}}$.
    \end{Rem}
     Denote the quotient cohomology ring
    \begin{equation}\overline{H}^i(Y_t, \mathbb{Q}) := H^i(Y_t,\mathbb{Q})/[Alg^i(Y_t)+R^i(Y_t)],\end{equation} 
    in analogy to $\overline{H}^i(\mathcal{M}_t, \mathbb{Q})$ appearing in Proposition \ref{surjectivity}.
      Let 
      
      $$g_{Y_t, i}: H^{4n-i}(Y_t, \mathbb{Q})\longrightarrow H^i(Y_t, \mathbb{Q})$$ be the map induced by the restriction of the cycle given in (\ref{g}) to $Y_{t}\times Y_{t}$, and let \begin{equation}\label{overline g}\overline{g}_{Y_t, i}: H^{4n-i}(Y_t, \mathbb{Q})\longrightarrow \overline{H}^i(Y_t, \mathbb{Q})\end{equation} denote the composition of $g_{Y_t, i}$ with the surjection 
      \begin{equation}\label{surjection onto the quotient vector space}
      H^i(Y_t, \mathbb{Q})\longrightarrow \overline{H}^i(Y_t, \mathbb{Q}).
      \end{equation}
      Analogously, define \begin{equation}\label{overline f}\overline{f}_{Y_t, i}: H^{4n-i}(Y_t, \mathbb{Q})\longrightarrow \overline{H}^i(Y_t, \mathbb{Q})\end{equation} as the composition of the restriction of the cycle (\ref{restriction of f}) to $Y_t\times Y_t$ and the surjection (\ref{surjection onto the quotient vector space}).

\begin{Prop}\label{surjectivity of the restriction maps}
   Let $Y_t$ denote a variety of generalized Kummer type. If $k_0$ is a positive integer for which the restriction homomorphism $i^*$ is surjective for all $k$ in the range $0\leq k \leq k_0$, then the images of $\overline{f}_{Y_t, k}$ and $\overline{g}_{Y_t, k}$, given by (\ref{overline f}) and (\ref{overline g}) respectively, coincide for each such $k\geq 1$ and $\overline{g}_{Y_t, k}$ is surjective onto $\overline{H}^k(Y_t, \mathbb{Q})^{\Gamma}$ for each such $k\geq 2$.
\end{Prop}
\begin{proof}
    By Corollary \ref{f and g}, the maps $\overline{f}_k$ and $\overline{g}_k$ coincide for $k\geq 2$. For $k$ in the range $0\leq k\leq k_0$ the restriction homomorphism is surjective and it is clear the restrictions $\overline{f}_{Y_t, k}$ and $\overline{g}_{Y_t, k}$ coincide as well. 
 By Corollary \ref{f and g}, $\overline{g}_k$ is surjective in degree $\geq 2$, so therefore in degrees $2\leq k\leq k_0$, $\overline{g}_{Y_t, k}$ is surjective.
\end{proof}

\begin{Th}\label{LSC for Kummers}
    Let $Y_t$ be a variety of generalized Kummer type. If $k_0$ is a positive integer for which the restriction homomorphism $i^*: H^k(\mathcal{M}_t, \mathbb{Q})\rightarrow H^k(Y_t, \mathbb{Q})$ is surjective for all $k$ in the range $0\leq k\leq k_0$, then the Lefschetz standard conjecture holds for $Y_t$ in degrees $0\leq k\leq k_0$. 
\end{Th}
\begin{proof}
The proof is nearly identical to that of Theorem \ref{LSCforM}. Let $E^1_{Y_{t_1}}$ denote the restriction of $E^1$ to $Y_{t_1}\times Y_{t_1}$. Suppose $k$ is in the range $0\leq k\leq k_0$. Choose $C$ as in Proposition \ref{twistor family} and let $t_0, t_1\in \mathfrak{M}_{w^{\perp}}^0$ be points parametrizing $Kum_n(A)$ and $Y_{t_1}$ respectively. Proposition \ref{twistor family} implies there is a flat section $\tilde{\sigma}$ of the local system $R^*\Tilde{p}_*\mathbb{Q}$ through the class 
$$
-\theta_{k}(E^1)_K\in H^k(Kum_n(A)\times Kum_n(A))
$$
where $-\theta_{k}(E^1)_K$ is defined as the composition of $-\theta_k(E^1)$ from (\ref{theta class 2}) with inclusion, and which is algebraic in $H^k(Y_{t_1}\times Y_{t_1})$. Then there is a flat section $\Tilde{\rho}$ of the local system $R^*\Tilde{p}_*\mathbb{Q}$ through $g_K$ and algebraic in $H^*(Y_{t_1}\times Y_{t_1})$. We prove the theorem by induction. The result is clear for $k = 0$. Assume the LSC holds for $Y_{t_1}$ in degrees up to $k-1$, for $k\leq k_0$. The correspondence 
$$
\Tilde{\rho}_i(t_2): H^{4n-k}(Y_{t_1})\longrightarrow H^k(Y_{t_1})
$$
is the fiber of the morphism of local systems 
$$
R^{4n-i}p_*\mathbb{Q}\longrightarrow R^ip_*\mathbb{Q}.
$$
The morphism $\overline{g}_{Y_{t_1}, k}$ is surjective for $k\geq 2$ by Proposition \ref{surjectivity of the restriction maps} and as a result $\Tilde{\rho}_k(t_1)$ is surjective for $k\geq 2$. The result therefore follows in degree $k$ from Corollary \ref{CharlesCorollary}. The first cohomology of $Y_{t_1}$ vanishes and hence $H^k(Y_{t_1}, \mathbb{Q})$ satisfies the LSC in degrees $0\leq k\leq k_0$ by induction.
\end{proof}

Note the following immediate application:

\begin{Cor}
For $Y$ a variety of generalized Kummer deformation type, the Lefschetz standard conjecture holds in degrees $2$ and $3$.
\end{Cor}
\begin{proof}
By Lemma \ref{deformation invariant}, $\Gamma$ deforms as a local system in $\mathfrak{M}_{w^{\perp}}^0$, and $H^2(Y, \mathbb{Q})$ and $H^3(Y, \mathbb{Q})$ are both contained in $H^*(Y, \mathbb{Q})^{\Gamma}$. Hence, the restriction homomorphism $i^*$ is surjective in these degrees and the result follows from the previous theorem.
\end{proof}

\vspace{5mm}
\subsection{The decomposition theorem and proofs of the main results}\label{The decomposition theorem and proofs of the main results}
 The results of the previous subsection suggest that we must determine the degrees for which the pullback of the inclusion 
 $$
 i^*: H^{*}(\mathcal{M}(v), \mathbb{Q})\longrightarrow H^*(Kum_a(v), \mathbb{Q})
 $$
 is surjective, i.e. for what degrees $i$ do we have $H^i(Kum_{a}(v), \mathbb{Q}) = H^i(Kum_a(v), \mathbb{Q})^{\Gamma}$. This requires us to describe the cohomology of a generalized Kummer variety in terms of a relative stratification of the symmetric product $A^{(n+1)}$. The cohomological decomposition theorem in the form of \cite[Theorem 7]{GS} and the interaction of the $\Gamma$-action with the relative strata play a crucial role here. Let us set up the situation geometrically:

Let $h:A^{[n+1]}\rightarrow A^{(n+1)}$ denote the Hilbert-Chow morphism, $\sigma: A^{(n+1)}\rightarrow A$ the summation map, and $s_{n+1}: A^{[n+1]}\rightarrow A$ the Albanese morphism (\ref{albanese}). Define \begin{equation}\label{W}W: = \sigma^{-1}(0).\end{equation} We have the following commutative diagram:

\vspace{4mm}
\begin{center}
\begin{tikzcd}
Kum_n(A) \arrow[r, "i"', hook] \arrow[rrd, "0", dashed, bend right] \arrow[rrrr, "b", bend left] & {A^{[n+1]}} \arrow[rr, "h"] \arrow[rd, "s_{n+1}"'] &   & A^{(n+1)} \arrow[ld, "\sigma"] & W \arrow[l, "j", hook'] \arrow[lld, "0", dashed, bend left] \\
                                                                                     &                                                & A &                           &                                                       
\end{tikzcd}
\end{center}
\vspace{4mm}
Note that the singular locus of $W$ is precisely the intersection of $W$ with the singular locus of $A^{(n+1)}$. Note also that $A^{(n+1)}$ is stratified with respect to the set $P(n+1)$ of partitions of $n+1$. We call the stratification of $A^{(n+1)}$ with respect to $P(n+1)$ the \textit{Hilbert-Chow stratification} and we shall denote a stratum of the Hilbert-Chow stratification by $A^{(n+1)}_{\nu}$, for $\nu\in P(n+1)$. The Hilbert-Chow stratification expresses $A^{(n+1)}$ as the the disjoint union of locally closed strata
$$
A^{(n+1)} = \coprod_{\nu\in P(n+1)} A^{(n+1)}_{\nu}.
$$
We denote $W_{\nu}^{(n+1)}: = A^{(n+1)}_{\nu}\cap W$ and $K_{\nu}^{[n+1]}: = b^{-1}(W_{\nu}^{(n+1)})$.

\begin{Rem}
    Observe that $W$ is stratified with respect to the disjoint union of locally closed strata $$W = \coprod_{\nu\in P(n+1)}W^{(n+1)}_{\nu}.$$
    A point $y\in K_{\nu}^{[n+1]}$ in the fiber of the morphism $b_{\nu}$ may be described as $y = (Z_1, ..., Z_r)$, where for $\nu = (n_1, ..., n_r)$, $Z_i\subseteq Kum_n(A)$ is a length-$n_i$ subscheme of $A$ supported set-theoretically on the closed point $y_i\in A$. The points $y_i$ have the property that $\sum_i n_i y_i = 0$ in $A$. Furthermore, there are Cartesian diagrams
\begin{center}
    \begin{tikzcd}
{K_{\nu}^{[n+1]}} \arrow[d, "b_{\nu}"'] \arrow[r, "i_{\nu}"] & Kum_n(A) \arrow[d, "b"] \\
W_{\nu}^{(n+1)} \arrow[r, "j_{\nu}"]                         & W                      
\end{tikzcd}
\end{center}
for each $\nu\in P(n+1)$, where $i_{\nu}$ and $j_{\nu}$ are inclusion maps. Note that the morphism $b_{\nu}: K^{[n+1]}_{\nu}\rightarrow W_{\nu}^{(n+1)}$ has connected fibers and hence induces a bijection between the connected components of $K^{[n+1]}_{\nu}$ and those of $W_{\nu}^{(n+1)}$. 
\end{Rem}

\begin{Ex}\label{smallest partition}
As an example, let us consider $\nu = (n+1)$, the smallest partition of $n+1$. The stratum $A^{(n+1)}_{(n+1)}$ is isomorphic to $A$. The corresponding stratum $$W_{(n+1)}^{(n+1)} = W\cap A^{(n+1)}_{(n+1)}$$ consists of the $(n+1)^4$ points which in turn correspond to the torsion points of $A$ of order dividing $n+1$. The fiber $$K^{[n+1]}_{(n+1)} = b^{-1}(W_{(n+1)}^{(n+1)})$$ consists of Lagrangian subvarieties of $Kum_n(A)$. Since the points of $W_{(n+1)}^{(n+1)}$ are in a connected stratum of $A^{(n+1)}$, these Lagrangian fibers are pairwise homologous in $A^{[n+1]}$. 
\end{Ex}

\begin{Def}\label{partition definition}
Let $\nu\in P(n+1)$ be a partition. We can express $\nu$ as $$n+1 =  \nu_1\cdot 1+ \nu_2\cdot 2+ \cdots +\nu_r\cdot r,$$ where $\nu_i$ denotes the number of times $i$ appears in the partition. Define
$$
A^{(\nu)}: = A^{(\nu_1)}\times \cdots \times A^{(\nu_r)}, 
$$
where $A^{(\nu_i)}$ denotes the $\nu_i$-th symmetric product of $A$. 
\end{Def}

Consider the collection of maps of sets
$$
g: A\longrightarrow \mathbb{Z}_{\geq 0}
$$
defined on closed points of $A$. We may choose to define the symmetric product of $n+1$ copies of $A$ as the following subset of this collection: 
$$
A^{(n+1)}: = \{g: A\rightarrow \mathbb{Z}_{\geq 0}: \sum_a g(a) = n+1\}.
$$
We may read off a parition $\nu$ given by $n+1 = \nu_1\cdot 1+ \nu_2\cdot 2+ \cdots + \nu_r\cdot r$ in terms of the maps $g$ in the collection $A^{(n+1)}$: For each $k$, $\nu_k$ corresponds to the number of closed points $a\in A$ such that $g(a) = k$. We may further identify 
$$
A^{(n+1)}_{\nu} = \{g\in A^{(n+1)}: |g^{-1}(i)| = \nu_i, \forall i\},
$$
where $|g^{-1}(i)|$ denotes the cardinality of the fiber. Thus, we can recognize the stratum $A^{(n+1)}_{\nu}$ to which a particular $g$ belongs. For $A^{(\nu)}$ in Definition \ref{partition definition}, there is an embedding $$p_{\nu}: A^{(\nu)}\hookrightarrow \overline{A^{(n+1)}_{\nu}}$$ which is given by $(g_1, ..., g_r)\mapsto \sum_i ig_i$.

Define $W_{\nu}\subset A^{(\nu)}$ by
\begin{equation}\label{W_nu}
W_{\nu}: = \{(g_1, ..., g_r)\in A^{(\nu)}: \sum_{i, a} ig_i(a)\cdot a = 0\}. 
\end{equation}
There exists an associated embedding $q_{\nu}: W_{\nu}\hookrightarrow \overline{W}^{(n+1)}_{\nu}$, again defined by $(g_1, ..., g_r)\mapsto \sum_i ig_i$, in analogy to the embedding $p_{\nu}$. The map $q_{\nu}$ relates the strata $W_{\nu}$ to the strata $W^{(n+1)}_{\nu}$ in the Hilbert-Chow decomposition. Indeed, there are Cartesian diagrams 
\begin{center}
    \begin{tikzcd}
W_{\nu} \arrow[r] \arrow[d, "q_{\nu}"] & A^{(\nu)} \arrow[d, "p_{\nu}"] \\
\overline{W_{\nu}^{(n+1)}} \arrow[r]   & \overline{A_{\nu}^{(n+1)}}  
\end{tikzcd}
\end{center}
\begin{Def}
Expressing a partition $\nu$ in terms of $n+1 = \nu_1\cdot 1 + \cdots + \nu_r\cdot r$, define 
\begin{equation}|\nu| : = \sum_i \nu_i.
\end{equation}
\end{Def}

\begin{Def}\label{definition of gcd}
 For a partition $\nu$ expressed as $n+1 = \nu_1\cdot 1+ \nu_2\cdot 2 + \cdots + \nu_r\cdot r$, define 
 $$d(\nu): = \mathrm{gcd}\{i: \nu_i\neq 0\}.$$ 
 Equivalently, if we express $\nu\in P(n+1)$ as $\nu = (n_1, ..., n_k)\in P(n+1)$, then $d(\nu): = \mathrm{gcd}(n_1, ..., n_k)$.  Let 
 \begin{equation}A[d(\nu)]\end{equation} be the group of $d(\nu)$-torsion points of $A$.
 \end{Def}

The Hilbert-Chow morphism $h: A^{[n+1]}\rightarrow A^{(n+1)}$, and likewise its restriction $b$, are proper morphisms that contract the divisor of non-reduced subschemes of $A^{[n+1]}$. We are interested in studying the cohomology of the generalized Kummer via this contraction. The crucial theorem we utilize for this purpose is \cite[Theorem 7]{GS}, which by means of the cohomological decomposition theorem describes the cohomology of $Kum_n(A)$ in terms of the morphism $b: Kum_n(A)\rightarrow W$ and the Hilbert-Chow stratification of $W$.

 \begin{Th}({\cite[ Theorem 7]{GS}}\label{GS Theorem 7})
There is a canonical isomorphism of Hodge structures 
\begin{equation}\label{GS}
H^{i+2n+2}(A\times Kum_n(A), \mathbb{Q})(n+1) \cong \bigoplus_{\nu\in P(n+1)}\bigoplus_{t \in A[d(\nu)]} H^{i+2|\nu|}(A^{(\nu)}, \mathbb{Q})(|\nu|).
\end{equation}
\end{Th}

\noindent The proof of the above isomorphism relies on considering the intersection cohomology of the smooth, $(2n+2)$-dimensional complex variety $A\times Kum_n(A)$ and the intersection cohomology of the strata $A\times \overline{W_{\nu}^{(n+1)}}$ relative to the morphism $1_{A}\times b$. Precisely, the isomorphism (\ref{GS}) is induced by an isomorphism
$$
(1_A\times b)_*IC(A\times Kum_n(A))\cong \bigoplus_{\nu\in P(n+1)} IC(A\times \overline{W_{\nu}^{(n+1)})}.
$$

\begin{Rem}
    As a vector space, $H^{i+2n+2}(A\times Kum_n(A), \mathbb{Q})(n+1)$ is the same as $H^{i+2n+2}(A\times Kum_n(A), \mathbb{Q})$. The twist $(n+1)$ is present to modify the Hodge type by $-2(n+1)$. The same goes for the right-hand side of the isomorphism. Although for the remainder of this article we will not take the Hodge structure into consideration, we will retain both the notation and the statement of the theorem in the same way as it appears in \cite{GS}.
\end{Rem}

We study next how the $\Gamma$-action on $A^{(\nu)}$ yields a permutation action on the right-hand side of (\ref{GS}); the next several definitions and results build to Proposition \ref{Gamma equivariance of isomorphism}, which gives us an equivariant version of Theorem \ref{GS Theorem 7}.

 \begin{Rem}
 For each permutation $\nu\in P(n+1)$, the variety $W_{\nu}$ admits a stratification
$$
W_{\nu} = \coprod_{t\in A[d(\nu)]} W_{\nu}^{t},
$$
where 
$$
W_{\nu}^{t}: = \{(g_1, ..., g_r)\in A^{(\nu)}: \sum_{i, a}(i/d(\nu))g_i(a)\cdot a = t\}.
$$
\end{Rem}

\begin{Prop}\label{transitive gamma action}
 The group $A[d(\nu)] = m\Gamma\subset A$ acts transitively on the components $W^{t}_{\nu}$, for $t\in A[d(\nu)]$.
\end{Prop}
\begin{proof}

There is a commutative diagram
\begin{center}
    \begin{tikzcd}
\Gamma \arrow[d, hook] \arrow[r, "{m}"] & {A[d(\nu)]}  \arrow[d, hook]\\
A \arrow[r]                      & A                         
\end{tikzcd}
\end{center}
where $m$ denotes multiplication by $m$.
The action of $A$ on symmetric powers of $A$ induces the diagonal action of $\gamma\in \Gamma$ on $W_{\nu}^t$ defined by 
$$
\gamma(g_1, ..., g_r) = (\gamma g_1, ..., \gamma g_r)
$$
where for a closed point $a\in A$, $\gamma g_i(a) = g_i(a-\gamma)$. Such an action induces an action on the sum  
$$
\sum_{i, a}(i/d(\nu))g_i(a)\mapsto \sum_{i, a}(i/d(\nu))g_i(a-\gamma)
.$$
The translation action of $\Gamma$ on $A$ is transitive, therefore we can make the substitution $a^{\prime} = a-\gamma$. Now we have 
$$
\sum_{i, a}(i/d(\nu))g_i(a)\cdot a = t,
$$
while
$$
\sum_{i, a}(i/d(\nu))g_i(a-\gamma)\cdot a = \sum_{i, a}(i/d(\nu))g_i(a^{\prime})\cdot (a^{\prime}+\gamma) = t+ m\gamma.
$$
Therefore, the $\Gamma$-action on $W_{\nu}$ factors through an action of $m\Gamma$, and the action of $m\Gamma$ on the connected components $W_{\nu}^{t}$ sending $W_{\nu}^{t}\mapsto W_{\nu}^{t+m\gamma}$ can be thought of as a $\Gamma$-action on its $\Gamma/m\Gamma$-cosets and hence is transitive. 
\end{proof}

 We next relate the cohomology of the Hilbert-Chow stratification for partitions of $n+1$ to the cohomology of $A^{(\nu)}$.

\begin{Lem}\label{GS step 7}
    There exist canonical isomorphisms 
    $$
     H^i(A\times W_{\nu}^{t}, \mathbb{Q})\xrightarrow[]{\sim} H^i(A^{(\nu)}, \mathbb{Q})
    $$
    for each $\nu\in P(n+1)$ and each $t\in A[d(\nu)]$, induced by isomorphisms
    $$
A\times W_{\nu/d(\nu)}\xrightarrow{\sim} A\times W_{\nu}^{t}.
$$ Furthermore, the pullback action of $A[m]$ on $H^i(A\times W_{\nu}^t, \mathbb{Q})$ is trivial.
\end{Lem}

\begin{proof}
    The proof appears in arguments outlined in \cite[pp. 242-243]{GS}. We reproduce them here for the convenience of the reader. 
    
    Consider first the case where $d(\nu) = 1$. There is an action of $A$ on symmetric powers of $A$ given by $zg(a) = g(a-z)$. For $\gamma\in \Gamma$ and $a\in A$ there is a corresponding diagonal action of $\Gamma$ on closed points of $A\times W_{\nu}$ given by 
\begin{equation}\label{anti-diagonal action}
\gamma(a, g_1,\dots , g_r) = (a-\gamma, \gamma g_1, ..., \gamma g_r)
\end{equation}
for a subscheme $(g_1, \dots, g_r)\in W_{\nu}$. The map $A\times W_{\nu}\rightarrow A^{(\nu)}$ is the quotient by this action. It suffices to prove that the diagonal action of $A[m]$ is trivial on the cohomology of $A\times W_{\nu}$ in this case. Let $\Tilde{A}^{\nu}$ denote the Galois cover of $A^{(\nu)}$. $\Tilde{A}^{\nu} = A^{\nu_1}\times \cdots\times  A^{\nu_r}$ and has associated Galois group $\mathcal{S}_{\nu} = \mathcal{S}_{\nu_1}\times \cdots \times \mathcal{S}_{\nu_r}$. Let 
$$
s: \Tilde{A}^{\nu}\longrightarrow A
$$
be defined on closed points by 
$$
((y_j^1)_{j = 1}^{\nu_1}, \dots , (y_j^r)_{j =1}^{\nu_r})\mapsto \sum_{i, j} iy_j^i
$$
and define $\tilde{W}_{\nu}: = \mathrm{ker}(s)$. Now the assumption that $d(\nu) = 1$ implies that we may write $1  = \sum iz_i$ for $z_i\in \mathbb{Z}$. There is therefore a splitting of $s$ and we may conclude that $\Tilde{A}^{\nu} \cong \tilde{W}_{\nu}\times A$. The diagonal action of $\Gamma$ on $A\times W_{\nu}$ induces a diagonal action on $A\times \Tilde{W}_{\nu}$ given by 
$$
\gamma(a, (y_j^1)_{j = 1}^{\nu_1}, \dots , (y_j^r)_{j =1}^{\nu_r}) = (a-\gamma, \tau_{\gamma}(y_j^1), \dots ,  \tau_{\gamma}(y_j^r))
$$
where $\tau$ denotes the translation map. The morphism $A\times \Tilde{W}_{\nu}\rightarrow A\times W_{\nu}$ is $\Gamma$-equivariant and is realized by the quotient of the action by $\mathcal{S}_{\nu}$ and there is therefore a $\Gamma$-equivariant isomorphism 
$$
H^{*}(A\times W_{\nu})\cong H^*(A\times \Tilde{W}_{\nu})^{\mathcal{S}_{\nu}}.
$$
$A\times \Tilde{W}_{\nu}$ acts on itself and there is an embedding of the group $\Gamma$ into $A\times \Tilde{W}_{\nu}$ given by 
$$
\gamma\mapsto (-\gamma, (\gamma)_{1}^{\nu_1}, \dots , (\gamma)_1^{\nu_r}). 
$$
The action of $\Gamma$ is isotopic to the identity and hence is the identity on $H^*(A\times \Tilde{W}_{\nu})$.

    Assume now that $d(\nu)>1$. For a partition $\nu$ defined by $$n+1 = \nu_1\cdot 1+ \nu_2\cdot 2+ \cdots + \nu_r\cdot r,$$ denote the partition $\nu/d(\nu)\in P(\frac{n+1}{d(\nu)})$ by
$$
\frac{n+1}{d(\nu)} = \frac{1}{d(\nu)}\cdot \nu_1+ \frac{2}{d(\nu)}\cdot \nu_2+ \cdots + \frac{r}{d(\nu)}\cdot \nu_r.
$$
Let $z\in A$ be a point with the property that $((n+1)/d(\nu))z = t$, for $t\in A[d(\nu)]$. We define a morphism 
$$
j_z: W_{\nu/d(\nu)}\longrightarrow W_{\nu}^{t}
$$
where we send a point $(g_1, ..., g_{r/d(\nu)})\mapsto (h_1, ..., h_r)$, for which $h_i = 0$ unless $i$ is a multiple of $d(\nu)$. Furthermore, we set $h_{d(\nu)i}= g_i(a-z)$, for all $i$ and for $a\in A$.

For a point $z$ with the property that $((n+1)/d(\nu))z = t$, for $t\in A[d(\nu)]$, there is an induced isomorphism
$$
A\times W_{\nu/d(\nu)}\longrightarrow A\times W_{\nu}^{t}
$$
given by the anti-diagonal action $(a, g)\mapsto (a-z, j_z(g))$. This isomorphism induces an isomorphism on cohomology
$$
H^*(A\times W_{\nu/d(\nu)}, \mathbb{Q})\xrightarrow{\sim} H^*(A\times W_{\nu}^t, \mathbb{Q})
$$
which is independent of $z$, since any two choices of $z$ differ by an element of $A[(n+1)/d(\nu)]$. We have $A^{(\nu)} = A^{\nu/d(\nu)}$, and the above isomorphism induces the desired isomorphism due to the fact that the group action of $A[m]$ on the cohomology of $A\times W_{\nu/d(\nu)}$ is trivial.
\end{proof}

\begin{Cor}\label{representation theory of the cohomology}
    For a fixed partition $\nu\in P(n+1)$, there is an isomorphism of $\mathbb{Q}[\Gamma]$-modules 
    $$
    \bigoplus_{t\in A[d(\nu)]} H^{i+2|\nu|}(A\times W_{\nu}^t, \mathbb{Q})(|\nu|) \cong H^{i+2|\nu|}(A\times W_{\nu}^0,\mathbb{Q})(|\nu|)\otimes B_{\nu},
    $$
    where $B_{\nu}$ is the pullback via $\Gamma\rightarrow m\Gamma$ of the regular $m\Gamma$-representation.
\end{Cor}
\begin{proof}
    The the action of $\Gamma$ on $W_{\nu}$ factors through an action of $m\Gamma$, which acts transitively via the diagonal action on the components $W^t_{\nu}$ by the previous proposition. The pullback action on cohomology is therefore a transitive permutation action of $m\Gamma$ on the vector spaces $H^{i+2|\nu|}(A\times W^{t}_{\nu}, \mathbb{Q})(|\nu|)$.  
\end{proof}

The next proposition strengthens the result of Theorem \ref{GS Theorem 7}:

\begin{Prop}\label{Gamma equivariance of isomorphism}
There is an isomorphism of $\mathbb{Q}[\Gamma]$-modules \begin{equation}\label{GS equivariant}
H^{i+2n+2}(A\times Kum_n(A), \mathbb{Q})(n+1) \cong \bigoplus_{\nu\in P(n+1)} [H^{i+2|\nu|}(A^{(\nu)}, \mathbb{Q})(|\nu|)\otimes B_{\nu}],
\end{equation}
 where $B_{\nu}$ is the pullback via the map $\Gamma\rightarrow m\Gamma$ of the $m\Gamma$-regular representation.
\end{Prop}
The proof requires that we check the equivariance properties of the isomorphism (\ref{GS}) with respect to the $\Gamma$-action on $A^{(\nu)}$ and on $A\times Kum_n(A)$. We check the equivariance in several steps, each of which correspond to a step required in the proof of \cite[Theorem 7]{GS}. Since $A\times Kum_n(A)$ is a smooth complex algebraic variety of dimension $2n+2$, the intersection complex is the constant sheaf of locally constant $\mathbb{Q}$-valued functions
$$
\underline{A\times Kum_n(A)}[2n+2].
$$
We hence have the isomorphism 
$$
 H^{i+2n+2}p_*\underline{A\times Kum_n(A)}(n+1) \cong H^{i}p_*IC(A\times Kum_n(A))(n+1),
$$
where $p$ denotes projection to a point. The isomorphism above respects the $\Gamma$-action; the sheaf $\underline{A\times Kum_n(A)}$ is an object in the category $Perv_{\Gamma}(A\times Kum_n(A))$ of $\Gamma$-equivariant perverse sheaves. Hence, its cohomologies may be realized as $\mathbb{Q}[\Gamma]$-modules.
In the following, the isomorphism of each step implicitly means an isomorphism of $\mathbb{Q}[\Gamma]$-modules:
\begin{itemize}
\item[\textbf{Step 1:}] (cf. \cite[pg. 241 (2-3)]{GS}) There exists an isomorphism
$$
H^{i}p_*IC(A\times Kum_n(A))(n+1) \cong \bigoplus_{\nu\in P(n+1)} H^ip_*IC(A\times \overline{W_{\nu}^{(n+1)}})(|\nu|).
$$
\begin{proof}
The above ismorphism arises by applying the functor $H^ip_*$ to the isomorphism 
\begin{equation}\label{hilbert chow intersection cohomology}
(1_A\times b)_*IC(A\times Kum_n(A)) \cong \bigoplus_{\nu\in P(n+1)}IC(A\times \overline{W_{\nu}^{(n+1)}}).
\end{equation}
It suffices to observe the $\Gamma$-equivariance of the above isomorphism, i.e. an isomorphism of $\Gamma$-equivariant objects, which we have by \cite[Theorem 5.3, pg. 42]{BL}. To briefly elaborate: The morphism $$1_A\times b: A\times Kum_n(A)\rightarrow A\times W$$ is a proper, $\Gamma$-equivariant morphism. For $\mathcal{F}\in D^b_{\Gamma}(A\times Kum_n(A))$ a semi-simple object, the pushforward $(1_A\times b)_*(\mathcal{F})$ is a semi-simple object in $D^b_{\Gamma}(A\times W)$. Here, $D^b_{\Gamma}(\cdot)$ denotes the equivariant derived category. The equivariant version of the equality (\ref{hilbert chow intersection cohomology}) follows, as the objects of $IC_{\Gamma}(\cdot)$ are precisely the simple objects of $Perv_{\Gamma}(\cdot)$. Hence, there is an equivariant isomorphism 
$$
(1_A\times b)_*IC_{\Gamma}(A\times Kum_n(A)) \cong \bigoplus_{\nu\in P(n+1)}IC_{\Gamma}(A\times \overline{W_{\nu}^{(n+1)}}).
$$
\end{proof}

\item[\textbf{Step 2:}] (cf. \cite[pg. 241 (3-4)]{GS}) There exists an isomorphism
$$
\bigoplus_{\nu\in P(n+1)}  H^{i}p_{*}IC(A\times \overline{W^{(n+1)}_{\nu}})(|\nu|)  \cong \bigoplus_{\nu\in P(n+1)}  H^{i}p_{*}q_{\nu, *}\underline{A\times W_{\nu}}[2|\nu|](|\nu|).
$$
\begin{proof}
    The morphism $1_{A}\times q_{\nu}: A\times W_{\nu}\hookrightarrow A\times \overline{W^{(n+1)}_{\nu}}$ is obviously $\Gamma$-equivariant with respect to the $\Gamma$-action on the symmetric product, so as a result it suffices to observe the $\Gamma$-equivariance of the isomorphism
    $$
    IC(A\times \overline{W_{\nu}^{(n+1)}}) \cong q_{\nu, *}\underline{A\times W_{\nu}}[2|\nu|].
    $$
    The morphism $1_A\times q_{\nu}$ is finite, birational, and proper-- in fact an open embedding. In particular, it is small, and so we may apply \cite[Theorem 5.3, pg. 42]{BL}. 
\end{proof}

     \item[\textbf{Step 3:}] (cf. \cite[pg. 241 (4-5)]{GS}) There exists an equality
$$
\bigoplus_{\nu\in P(n+1)}  H^{i}p_{*}q_{\nu, *}\underline{A\times W_{\nu}}[2|\nu|](|\nu|) = \bigoplus_{\nu\in P(n+1)}  H^{i+ 2|\nu|}(A\times W_{\nu}, \mathbb{Q})(|\nu|).
$$ 
\begin{proof}
 The equality holds by definition. 
\end{proof}

 \item[\textbf{Step 4:}] (cf. \cite[pg. 241 (5-6)]{GS}). There is an isomorphism

     $$\bigoplus_{\nu\in P(n+1)}  H^{i+ 2|\nu|}(A\times W_{\nu}, \mathbb{Q})(|\nu|) \cong \bigoplus_{\nu\in P(n+1)} \bigoplus_{t\in A[d(\nu)]} H^{i+ 2|\nu|}(A\times W^t_{\nu}, \mathbb{Q})(|\nu|). $$
    \begin{proof}
        The above isomorphism exists as a result of the disjoint union 
        $$
        W_{\nu} = \coprod_{t\in A[d(\nu)]} W_{\nu}^t,
        $$ for each $\nu\in P(n+1)$.
        The right-hand side is $\mathbb{Q}[\Gamma]$-isomorphic to 
        $$
         \bigoplus_{\nu\in P(n+1)} [H^{i+ 2|\nu|}(A\times W^0_{\nu}, \mathbb{Q})(|\nu|)\otimes B_{\nu}],
        $$
        by Corollary \ref{representation theory of the cohomology}. The action of $\Gamma$ on $W_{\nu}$ factors through an action of $m\Gamma$, and so we get an isomorphism of $\mathbb{Q}[\Gamma]$-modules.
    \end{proof}

        \item[\textbf{\textbf{Step 5:}}] (cf. \cite[pg. 241 (6-7)]{GS}). There is an isomorphism
    $$
    \bigoplus_{\nu\in P(n+1)} \bigoplus_{t\in A[d(\nu)]} H^{i+ 2|\nu|}(A\times W^t_{\nu}, \mathbb{Q})(|\nu|) \cong \bigoplus_{\nu\in P(n+1)} \bigoplus_{t\in A[d(\nu)]} H^{i+ 2|\nu|}(A^{(\nu)}, \mathbb{Q})(|\nu|).
    $$
    \begin{proof}
    By the proof of Lemma \ref{GS step 7}, there are canonical isomorphisms 
    $$
    H^i(A^{(\nu)}, \mathbb{Q}) \cong H^i(A\times W_{\nu}^t, \mathbb{Q}),
    $$
    for each $\nu\in P(n+1)$ and $t\in A[d(\nu)]$, induced by isomorphisms
    $$
    A\times W_{\nu}^t \longrightarrow A^{(\nu)}
    $$
    which themselves act trivially on cohomology. There is therefore a $\Gamma$-action on $A^{(\nu)}$ which factors through an action of $m\Gamma$. So as a result, 
    $$
\bigoplus_{\nu\in P(n+1)} [H^{i+2|\nu|}(A^{(\nu)}, \mathbb{Q})(|\nu|)\otimes B_{\nu}] \cong \bigoplus_{\nu\in P(n+1)} [H^{i+2|\nu|}(A\times W_{\nu}^0, \mathbb{Q})(|\nu|)\otimes B_{\nu}], 
    $$
    for $B_{\nu}$ the pullback via the map $\Gamma\rightarrow m\Gamma$ of the $m\Gamma$-regular representation. 
    \end{proof}

\end{itemize}

\begin{Ex}
Let $\nu = (n+1)$, so that $d(\nu) = n+1$. The group $\Gamma$ is isomorphic to $A[n+1]$. There is a transitive $\Gamma$-action on the elements $W^{\gamma}_{(n+1)}$ which yields a $\Gamma$-regular representation 
$$
\bigoplus_{\gamma\in \Gamma}H^{2n+2}(A\times W^{\gamma}_{(n+1)}, \mathbb{Q}),
$$
and by the equivariance properties of the isomorphism (\ref{GS}), there is a $\Gamma$-regular representation in the middle cohomology 
$$
H^{2n+2}(A\times Kum_n(A), \mathbb{Q}).
$$
\end{Ex} 

Let us now proceed to the proof of the main results. We start with an easy lemma:

\begin{Lem}\label{smallest factor}
    Let $j>1$ be the smallest factor dividing $n+1$. If $|\nu|>\frac{n+1}{j}$, then $d(\nu) = 1$.
\end{Lem}
\begin{proof}
    Suppose for a contradiction that $d(\nu)>1$. We may write 
    $$
    n+1 = \sum_i i \nu_i = d(\nu)\sum_i\beta_i,
    $$
    where $\beta_i\geq \nu_i$ for each $i$. Recall that $|\nu| = \sum_i \nu_i$, so by assumption 
    $$
    j\sum_i \nu_i > d(\nu)\sum_i\beta_i.
    $$
    However, $d(\nu)\geq j$ and $\sum_i \beta_i \geq \sum_i \nu_i$, therefore we have a contradiction.
\end{proof}

\begin{Rem}
Proposition \ref{Gamma equivariance of isomorphism} and Lemma \ref{smallest factor} together imply that a necessary condition for $\Gamma$ to act trivially on the cohomology 
$$H^{i+2|\nu|}(A^{(\nu)}, \mathbb{Q})(|\nu|)$$ is that $|\nu|> \frac{n+1}{j}$. 
\end{Rem}

\begin{Lem}\label{surjective in degree}
    Let $t\in \mathfrak{M}_{v^{\perp}}^0$ be a point parametrizing the fibers $Y_{t}$ and $\mathcal{M}_{t}$. Let $j>1$ be the smallest prime divisor of $n+1$. The pullback map
    $$
    i^*: H^*(\mathcal{M}_t, \mathbb{Q})\longrightarrow H^*(Y_t, \mathbb{Q})
    $$
    is surjective in degrees $<2(n+1)(j-1)/j$.
\end{Lem}
\begin{proof}
 Proposition \ref{Gamma equivariance of isomorphism} and Lemma \ref{smallest factor} together imply that the summands of 
\begin{equation}\label{GS RHS}
\bigoplus_{\nu\in P(n+1)} \bigoplus_{t\in A[d(\nu)]} H^{i+2|\nu|}(A^{(\nu)}, \mathbb{Q})(|\nu|)
\end{equation}
which contribute to the complement of $H^k(A\times Kum_n(A), \mathbb{Q})^{\Gamma}$ in $H^k(A\times Kum_n(A), \mathbb{Q})$ are contained in the collection of those for which $|\nu|\leq \frac{n+1}{j}$.
By Theorem \ref{GS Theorem 7}, a necessary condition for the existence of summands of $H^k(A\times Kum_n(A), \mathbb{Q})$ in the complement of $H^k(A\times Kum_n(A))^{\Gamma}$ is that $|\nu|\geq -i/2$. Theorem \ref{GS Theorem 7} also implies that the contribution of (\ref{GS RHS}) to $H^k(A\times Kum_n(A))^{\Gamma}$ occurs in degree $k = i+2n+2$, so we have $i = k-2n-2$. Combining these inequalities yields the inequality
$$
i+\frac{2n+2}{j}\leq 0,
$$
and substituting $i = k-2n-2$ yields the inequality
$$
k-2n-2+\frac{2n+2}{j}\leq 0.
$$
The latter inequality is hence a necessary condition for the existence of non-$\Gamma$-invariant summands. Therefore, if  
$k<2(n+1)(j-1)/j$, then $H^k(A\times Kum_n(A), \mathbb{Q})$ is $\Gamma$-invariant and by the K\"{u}nneth theorem this implies that $H^{k}(Kum_n(A), \mathbb{Q})$ is $\Gamma$-invariant as well. $H^k(Y_t, \mathbb{Q})$ is therefore $\Gamma$-invariant by Lemma \ref{deformation invariant}.
\end{proof}

We can now prove Theorem \ref{maintheorem}:
\begin{Th}(Theorem \ref{maintheorem})\label{proof of main theorem}
 Let $Y$ be a variety of generalized Kummer deformation type of dimension $2n$ and let $j>1$ be the smallest prime divisor of $n+1$. The Lefschetz standard conjecture holds for $Y$ in degrees $<2(n+1)(j-1)/{j}$.
\end{Th}
\begin{proof}
By Lemma \ref{surjective in degree}, the pullback map $i^*: H^k(\mathcal{M}_t, \mathbb{Q})\rightarrow H^k(Y_t, \mathbb{Q})$ is surjective for degrees $k<2(n+1)(j-1)/j$, as the image of $i^*$ is $H^k(Y_t, \mathbb{Q})^{\Gamma}$. Therefore, the Lefschetz standard conjecture is satisfied for $Y_t$ in degree $k$ by Theorem \ref{LSC for Kummers} via the correspondence (\ref{restriction of f}).
\end{proof}

For what remains we shall perform an analysis of the case where $n+1$ is prime. In this case the only partition such that $d(\nu) \neq 1$ is $\nu = (n+1)\in P(n+1)$, for which $|\nu|= 1$, and so $d(\nu) = n+1$. Note that we have 
$$\Gamma\cong A[d(\nu)]\cong A[n+1].$$
We may write the direct summand corresponding to the partition $(n+1)$ on the right-hand-side of (\ref{GS}) as
$$
\bigoplus_{\gamma\in \Gamma} H^{i+2|\nu|}(A\times W_{(n+1)}^{\gamma}, \mathbb{Q})(|\nu|).
$$
Recall that as described in Example \ref{smallest partition}, $b^{-1}(W^{(n+1)}_{(n+1)}) = K^{[n+1]}_{(n+1)}$ is a collection of Lagrangian subvarieties in the middle cohomology of $Kum_n(A)$.
\noindent
It is then clear that in the middle cohomology there are $(n+1)^4$ direct summands in the right-hand side of (\ref{GS}), corresponding to the $(n+1)^4$ torsion points of $A$ which in turn correspond to elements of $\Gamma$, since $\Gamma$ is isomorphic to $A[n+1]$. Indeed, fixing $\nu = (n+1)$, we may rephrase the contents of Proposition \ref{Gamma equivariance of isomorphism} in this case as

\begin{Prop}\label{regular representation}
The direct sum
$$
 \bigoplus_{\gamma\in \Gamma}H^{2n+2}(A\times W_{(n+1)}^{\gamma}, \mathbb{Q})
$$
is a $\Gamma$-regular representation. The equivariance properties of the isomorphism (\ref{GS}) yields a $\Gamma$-regular representation in the middle cohomology of $A\times Kum_n(A)$. What's more, in degree $k \neq 2n+2$, the pullback $\Gamma$-action on $H^k(A\times Kum_n(A), \mathbb{Q})$ is trivial.
\end{Prop}
\begin{proof}
Since $\Gamma$ acts transitively on itself by translation, the result follows from Proposition \ref{transitive gamma action} and Proposition \ref{Gamma equivariance of isomorphism}.
\end{proof}

\begin{Rem}
Let $Y_t$ be a variety of generalized Kummer type parametrized by a point $t\in\mathfrak{M}_{w^{\perp}}^0$, and assume that $Y_t$ is of dimension $2n$ for which $n+1$ is prime. By Lemma \ref{surjective in degree}, the image of $i^{*}: H^k(\mathcal{M}_t, \mathbb{Q})\rightarrow H^k(Y_t, \mathbb{Q})$ is surjective in all degrees $k<2n$. In the middle degree $k = 2n$, Theorem \ref{GS Theorem 7} and Proposition \ref{regular representation} imply that $H^{2n}(Y_t, \mathbb{Q})$ is spanned by a $\Gamma$-regular representation. The prior classes are $\Gamma$-invariant, by Corollary \ref{generate translation invariant subring}. The cycles that appear in the $\Gamma$-regular representation satisfy the LSC in the middle cohomological degree via the correspondence $[\Delta]^*$, where $\Delta\subset Y_t\times Y_t$ is the class of the diagonal. The correspondence $[\mathcal{Z}]$ in the formulation of the LSC is the class $[\Delta]\in CH^*(Y_t\times Y_t)$ in middle cohomological degree and the class of the self-adjoint correspondence (\ref{restriction of f}) in degrees $<2n$.
\end{Rem}

For $\nu$ any partition besides $(n+1)$, the direct summand $H^i(A\times W_{\nu}^{(n+1)}, \mathbb{Q})$ is $\Gamma$-invariant by Proposition \ref{Gamma equivariance of isomorphism}, so Corollary \ref{maincorollary} follows immediately as a result of Theorem \ref{maintheorem}.

\vspace{5mm}

\textbf{Acknowledgements:} It is my pleasure to acknowledge my advisor Eyal Markman for many enlightening conversations and for all of his support. I graciously thank Tom Braden for helpful conversations on intersection cohomology. I thank Paul Hacking and Katrina Honigs for useful discussions regarding this project and Wendelin Lutz for reading a previous version and providing comments. I sincerely thank the anonymous referees for their helpful comments and suggestions.


\begin{thebibliography}{flushleft}


\bibitem[]{ACLS}Ancona, G., Cavicchi, M., Laterveer, R., Sacc\`{a}, G.: \textit{Relative and absolute Lefschetz standard conjectures for some Lagrangian fibrations.} Preprint,  arXiv:2304.00978v1.

\bibitem[]{Andre} Andr\'{e}, Y.: \textit{Une introduction aux motifs (motifs purs, motifs mixtes, périodes)}, volume 17
of Panoramas et Synthèses [Panoramas and Syntheses]. Société Mathématique de France,
Paris, 2004.



\bibitem[]{Ar} Arapura, D.: \textit{Motivation for Hodge cycles}. Adv. Math. 207 (2006), no. 2, 762–781.

\bibitem[]{At} Atiyah, M. F.: \textit{Vector Bundles and the K\"{u}nneth Formula}. Topology 1 (1962), 245-248.

\bibitem[]{AH} Atiyah, M., Hirzebruch, F.: \textit{Vector bundles and homogeneous spaces}. London Mathematical Society Lecture Note Series. Cambridge University Press, (1972), 196-222. DOI:10.1017/CBO9780511662584.020.


\bibitem[]{Beauville} Beauville, A.: \textit{Varietes K\"{a}hleriennes dont la premiere classe de Chern est nulle}. J. Diff. Geom. 18  (1983),
755-782.

\bibitem[]{BL} Bernstein, J., Lunts, V.: \textit{Equivariant Sheaves and Functors}. LNM 1578, Springer-Verlag, 1994.



\bibitem[]{Caldararu} C\u{a}ld\u{a}raru, A.: \textit{Derived categories of twisted sheaves on Calabi-Yau manifolds}. Thesis, Cornell Univ.,
May 2000.

\bibitem[]{C} Charles, F.: \textit{Remarks on the Lefschetz standard conjecture and hyperk\"{a}hler varieties}. Commentarii Mathematici Helvetici, vol. 88 (2010), 449-468.

\bibitem[]{CM}
Charles, F., Markman, E.: \textit{The standard conjectures for holomorphic symplectic varieties deformation equivalent to Hilbert schemes of K3 surfaces}. Compositio Mathematica, 149 (2011), no. 3, 481-494. DOI:10.1112/S0010437X12000607.


\bibitem[]{Chevalley} Chevalley, C.: \textit{The Algebraic Theory of Spinors and Clifford Algebras}. Springer-Verlag, 1997. 


\bibitem[]{de Cataldo, Migliorini}
de Cataldo, M., Migliorini, L.: \textit{The decomposition theorem, perverse sheaves and the topology of algebraic maps}.
Bull. Amer. Math. Soc. (N.S.) 46 (2009), no. 4, 535–633.


\bibitem[]{dJP} de Jong, A.J., Perry, A.: \textit{The Period-Index Problem and Hodge Theory}. Preprint, arXiv:2212.12971v1.



\bibitem[]{GKLR} Green, M., Kim, Y., Laza, R., Robles, C.: \textit{The LLV decomposition of hyper-K\"{a}hler cohomolog}y. Math. Ann. 382 (2022), no. 3-4, 1517–1590.

\bibitem[]{Gr} Grothendieck, A.: \textit{Standard conjectures on algebraic cycles}. Algebraic Geometry (Internat. Colloq.,
Tata Inst. Fund. Res., Bombay, 1968) pp. 193-199. Oxford Univ. Press, London, 1969.

\bibitem[]{Gottsche}  G\"{o}ttsche, L.: \textit{Hilbert Schemes of Zero-Dimensional Subschemes of Smooth Varieties}. LNM
1572, Springer-Verlag, 1994.

\bibitem[]{GS}
G\"{o}ttsche, L., Soergel, W.: \textit{Perverse sheaves and the cohomology of Hilbert schemes of smooth algebraic surfaces}. Math. Ann. 296 (1993), 235-245.



\bibitem[]{Hassett Tschinkel} Hassett, B., Tschinkel, Y.: \textit{Hodge theory and Lagrangian planes on generalized Kummer fourfolds}. Mosc. Math.
J. 13 (2013), no. 1, 33–56.



\bibitem[]{Huybrechts1} Huybrechts, D.: \textit{Compact Hyperk\"{a}hler Manifolds: Basic results}. Invent. Math. 135 (1999), no.
1, 63-113 and Erratum: Invent. Math. 152 (2003), no. 1, 209–212.

\bibitem[]{Hu2} Huybrechts, D.: \textit{Fourier-Mukai Transforms in Algebraic Geometry}. Oxford University Press,
2006.

\bibitem[]{Hu3}  Huybrechts, D.: \textit{A global Torelli theorem for hyperk\"{a}hler manifolds [after M. Verbitsky]}. S\'{e}minaire Bourbaki: Vol. 2010/2011. Ast\'{e}risque No. 348 (2012), Exp. No. 1040, 375–403.

\bibitem[]{HL} Huybrechts, D., Lehn, M.: \textit{The Geometry of Moduli Spaces of Sheaves}. Second edition. Cambridge University Press, Cambridge, 2010.

\bibitem[]{Karoubi} Karoubi, M.: \textit{K-theory: An Introduction}. Springer-Verlag, 1978.


\bibitem[]{K1} Kleiman, S.: \textit{The Standard Conjectures}. Proceedings of Symposia in Pure Mathematics, vol. 55 (1994), part 1, 3-20.

\bibitem[]{K2} Kleiman S.: \textit{Algebraic cycles and the Weil conjectures}, Dix expos\'{e}s sur la cohomologie des sch\'{e}mas,
North-Holland, Amsterdam, 1968, 359-386.




\bibitem[]{LL} Looijenga, E., Lunts, V.: \textit{A Lie algebra attached to a projective variety}. Invent. Math. 129 (1997), no. 2, 361–412.

\bibitem[]{Markman1} Markman, E.: \textit{On the monodromy of generalized Kummer varieties and algebraic cycles on their intermediate jacobians}. Journal of the European Mathematical Society (JEMS), DOI:10.4171/JEMS/1199.

\bibitem[]{Markman2} Markman, E.: \textit{The Beauville-Bogomolov class as a characteristic class}. J. Algebraic Geom.
29 (2020), no. 2, 199–245.

\bibitem[]{Markman3} Markman, E.: \textit{Generators of the cohomology ring of moduli spaces of sheaves on symplectic
surfaces.} J. Reine Angew. Math. 544 (2002), 61-82.

\bibitem[]{Markman4} Markman, E.: \textit{Rational Hodge isometries of hyper-k\"{a}hler
varieties of $K3^{[n]}$-type are algebraic.} Preprint, arXiv:2204.00516v2.


\bibitem[]{Markman5} Markman, E.: \textit{A survey of Torelli and monodromy results for holomorphic
symplectic varieties}. Complex and Differential Geometry, Springer Proc.
Math. 8, Springer, Heidelberg (2011), 257–322.

\bibitem[]{Markman6} Markman, E.: \textit{On the existence of universal families of marked irreducible holomorphic symplectic manifolds}. Kyoto
J. Math. 61 (2021), no. 1, 207–223.


\bibitem[]{Meachan}
Meachan, C.: \textit{Derived autoequivalences of generalised
Kummer varieties}. Math Res. Lett., vol. 22 (2015), no. 4, 1193-1221.

\bibitem[]{Mukai1} Mukai, S.: \textit{Symplectic structure of the moduli space of sheaves on an Abelian or K3 surface}.
Invent. math. 77 (1984), 101-116.

\bibitem[]{Mukai2} Mukai, S.: \textit{On the moduli space of bundles on K3 surfaces I, Vector bundles on algebraic
varieties}. Proc. Bombay Conference, 1984, Tata Institute of Fundamental Research Studies,
no. 11, Oxford University Press, 1987, pp. 341-413.

\bibitem[]{Mukai3} Mukai, S.: \textit{Fourier functor and its application to the moduli of bundles on an Abelian variety}.
Adv. Studies in Pure Math., 10 (1987), 515–550.

\bibitem[]{Munoz} M\~{u}noz, V.:\textit{ Spin(7)-instantons, stable bundles and the Bogomolov inequality for complex 4-
tori}. J. Math. Pures Appl. (9) 102 (2014), no. 1, 124–152.

\bibitem[]{Oguiso} Oguiso, K.: \textit{No cohomologically trivial non-trivial automorphisms of generalized Kummer
manifolds}. Nagoya Math. J. 239 (2020), 110–122.

\bibitem[]{Orlov} Orlov, D. O.: \textit{On equivalences of derived categories of coherent sheaves on Abelian varieties}.
Izv. Math. 66 (2002), no. 3, 569–594.


\bibitem[]{Verbitsky1} Verbitsky, M.: \textit{A global Torelli theorem for hyperk\"{a}hler manifolds}. Duke Math. J., vol. 162 (2013), no. 15, 2929-2986.

\bibitem[]{Verbitsky2} Verbitsky, M.: \textit{Hyperholomorphic sheaves and new examples of hyperk\"{a}hler manifolds},
alg-geom/9712012. In the book: Hyperk\"{a}hler manifolds, by Kaledin, D. and Verbitsky, M.,
Mathematical Physics (Somerville), 12. International Press, Somerville, MA, 1999.

\bibitem[]{Verbitsky3} Verbitsky, M.: \textit{Cohomology of compact hyper-K\"{a}hler manifolds and its applications}. Geom. Funct. Anal. 6 (1996), no. 4,
601–611.

\bibitem[]{Voisin1} Voisin, C.: \textit{The Hodge Conjecture}. In the book: Open Problems in Mathematics, edited by Nash, J.F. and Rassias, M., Springer, 2016.

\bibitem[]{Voisin2} Voisin, C.: \textit{On the Lefschetz standard conjecture for Lagrangian covered hyper-K\"{a}hler varieties}. Preprint,  arXiv:2007.11872v2.

\bibitem[]{Yoshioka} Yoshioka, K.: \textit{Moduli spaces of stable sheaves on Abelian surfaces}. Math. Ann. 321 (2001),
no. 4, 817–884.

\end{thebibliography}
\end{document}